\documentclass[amssymb,twoside,11pt]{article}
\usepackage{latexsym,amsmath,graphicx}
\usepackage[dvipsnames]{xcolor}
\usepackage[mathscr]{euscript}
\usepackage{placeins}
\usepackage{amsfonts}
\usepackage{enumerate}
\newtheorem{theorem}{Theorem}[section]

\newtheorem{proposition}[theorem]{{\bf Proposition}}
\newtheorem{cor}[theorem]{{\bf Corollary}}
\newtheorem{rem}[theorem]{{\bf Remark}}

\newtheorem{definition}{Definition}[section]
\numberwithin{equation}{section}
\newenvironment{proof}{\indent{\em Proof:}}{\quad \hfill
$\Box$\vspace*{2ex}}

\font\Bbb=msbm10 at 12pt

\newcommand{\R}{\mbox{\Bbb R}}
\newcommand{\N}{\mbox{\Bbb N}}

\usepackage{authblk}
\usepackage{mathtools}

\begin{document}
\title{On tempered fractional calculus with respect to functions and the associated fractional differential equations}

\author[1]{Ashwini D. Mali\thanks{Email: \texttt{maliashwini144@gmail.com}}}
\author[1]{Kishor D. Kucche\thanks{Email: \texttt{kdkucche@gmail.com}}}
\author[2]{Arran Fernandez\thanks{Email: \texttt{arran.fernandez@emu.edu.tr}}}
\author[2]{Hafiz Muhammad Fahad\thanks{Email: \texttt{hafizmuhammadfahad13@gmail.com}}}

\affil[1]{{\small Department of Mathematics, Shivaji University, Kolhapur-416 004, Maharashtra, India}}
\affil[2]{{\small Department of Mathematics, Eastern Mediterranean University, Famagusta, Northern Cyprus, via Mersin 10, Turkey}}

\maketitle

\def\baselinestretch{1.0}\small\normalsize

\begin{abstract}
The prime aim of the present paper is to continue developing the theory of tempered fractional integrals and derivatives of a function with respect to another function. This theory combines the tempered fractional calculus with the $\Psi$-fractional calculus, both of which have found applications in topics including continuous time random walks. After studying the basic theory of the $\Psi$-tempered operators, we prove mean value theorems and Taylor's theorems for both Riemann--Liouville type and Caputo type cases of these operators. Furthermore, we study some nonlinear fractional differential equations involving $\Psi$-tempered derivatives, proving existence-uniqueness theorems by using the Banach contraction principle, and proving stability results by using Gr\"onwall type inequalities.
\end{abstract}

\def\baselinestretch{1.5}
\allowdisplaybreaks

\section{Introduction}

The field of fractional calculus covers the study of integro-differential operators of fractional order and their applications in differential equations of various types and in the modelling of processes in physics, biology, etc. For broad overviews of these topics, we refer the reader to the well-known textbooks and survey articles such as \cite{Samko,Kilbas,Diethelm,Podlubny,Hilfer,sun-etal}.

Especially in recent years, fractional calculus has been developed to include a wide number of fractional integral and derivative operators. Generally, a new fractional derivative is proposed with one of two aims: either to model some physical processes for which the existing fractional derivatives are inadequate, or to capture some mathematical properties which the existing fractional derivatives lack. In terms of mathematical properties, the many different types of operators can generally be gathered into a few broad classes \cite{baleanu-fernandez}. Some of the important classes nowadays include the class of fractional integrals with analytic kernel functions and the associated fractional derivatives \cite{fernandez-ozarslan-baleanu}, the class of fractional integrals and derivatives with respect to a monotonic function \cite{Osler,Almeida}, and the class of fractional integrals and derivatives weighted by a multiplicative factor \cite{agrawal,ff:weighted}.

One type of fractional calculus of particular interest is the so-called tempered fractional calculus. Mathematically, it is of interest as the unique intersection between weighted fractional calculus and fractional calculus with analytic kernels \cite[Theorem 3.1]{ff:weighted}. In modelling, it has been used to understand turbulence in geophysical flows \cite{Meerschaert1}, and L\'evy processes such as Brownian motion \cite{Meerschaert2,Baeumera}. Tempered fractional calculus is so important that it has been rediscovered at least twice under different names, as generalised proportional fractional calculus \cite{jarad-abdeljawad-alzabut} and as substantial fractional calculus \cite{friedrich-jenko,Cao}.

We summarise as follows some of the important pure mathematical studies of tempered fractional calculus and the associated differential equations. Li et al. \cite{Li} examined various properties of tempered fractional derivatives and explored the existence, uniqueness, and stability of some tempered fractional ordinary differential equations. Zhao et al. \cite{L Zhao} demonstrated several properties involving function spaces and compositions, before considering variational calculus and spectral analysis for Riemann--Liouville-type tempered fractional differential equations. Morgado and Rebelo \cite{Morgado} considered nonlinear Caputo-type tempered fractional differential equations with a terminal condition, analysed existence and uniqueness properties of the solution, and proposed three numerical techniques to approximate solutions of the considered problems. Fernandez and Ustao\u{g}lu \cite{Fernandez} proved an analogue of Taylor's theorem for tempered fractional derivatives, and some integral inequalities, and utilised tempered operators to acquire some special functions such as hypergeometric and Appell's functions. Zaky \cite{Zaky} investigated the existence, uniqueness, and structural stability of solutions to nonlinear Caputo-type tempered fractional differential equations with generalised boundary conditions, and developed a spectral collocation technique for numerical solutions of the considered equations, including detailed convergence and error analysis.

The calculus of fractional integration and differentiation of a function with respect to another function, sometimes called $\Psi$-fractional calculus following the notation of Almeida \cite{Almeida}, is one of the three broad general classes of operators mentioned above. After the initial genesis of this idea in the work of Osler \cite{Osler} and Erd\'elyi \cite{Erdelyi}, it was further developed in the standard textbooks of Samko et al. \cite[\S18.2]{Samko} and Kilbas et al. \cite[\S2.5]{Kilbas}, all operators there being taken in the sense of Riemann--Liouville. The Caputo version of this established definition was formalised and studied by Almeida \cite{Almeida}, and the Hilfer version by Sousa and Oliveira \cite{Vanterler1}; the latter has been greatly promoted in the study of fractional differential equations \cite{kucche-mali-sousa,mali-kucche-sousa}. The broad class of $\Psi$-fractional operators includes, as special cases, the operators of Hadamard and Erd\'elyi, which had already excited interest before the formulation of the general class \cite{agrawal}.

It is interesting to note that both tempered fractional calculus and $\Psi$-fractional calculus have been found useful in the study of continuous time random walks. Several research investigations \cite{Sokolov,Schmidt,Cartea,Henry} have utilised tempered fractional integrals and derivatives without referring to them as such: for example, the transport operator $\mathscr{T}_t$ in \cite[Eq. (19)]{Sokolov} and \cite[Eq. (17)]{Schmidt}, used in studying reactions under anomalous diffusion of subdiffusive type, can clearly be identified as a tempered fractional derivative, while similar operators are also seen in \cite[Section III.A]{Cartea} and \cite[Section III.B]{Henry} in the analysis of continuous time random walks. On the other hand, the operators of fractional calculus with respect to an arbitrary monotonic function have also been used very recently \cite{wu1,wu2} in studying continuous time random walks. Thus, it is suggested that combining both ideas, tempered fractional calculus and $\Psi$-fractional calculus, may be a useful endeavour.

A recent paper of Fahad et al. \cite{Fahad} investigated tempered and Hadamard-type fractional calculi together, and the generalization of both which is given by taking the operators with respect to an arbitrary monotonic function. Such a generalisation can be thought of as $\Psi$-tempered fractional calculus, and it is a special case both of fractional calculus with analytic kernels with respect to functions \cite{Oumarou} and of weighted fractional calculus with respect to functions \cite{agrawal,ff:weighted}. The paper \cite{Fahad} established conjugation and limiting properties, function spaces and boundedness, and an integration by parts property, all in the context of $\Psi$-tempered fractional calculus (equivalently, Hadamard-type fractional calculus with respect to a function).

Our current work can be seen as the continuation of \cite{Fahad}, conducting a further investigation into $\Psi$-tempered fractional calculus and its properties. The structure of this paper is as follows. In Section \ref{preliminaries}, we recall definitions and basic results on the operators to be used, mostly from \cite{Fahad} and preceding works. In Section \ref{Sec:properties}, we consider various further properties of $\Psi$-tempered fractional integrals and derivatives: some of their limiting behaviours, their applications to some example functions, some composition relations between the operators, and the types of functions which have derivative zero in this model. In Section \ref{Sec:MVTTT}, we prove versions of the mean value theorem and Taylor's theorem in the setting of $\Psi$-tempered fractional calculus. In Section \ref{Sec:fde}, we study nonlinear fractional ordinary differential equations posed using $\Psi$-tempered derivatives of both Riemann--Liouville and Caputo types: firstly proving existence-uniqueness theorems for them, by constructing equivalent integral equations and using the Banach contraction principle, and then studying Ulam type stabilities by using a Gr\"onwall type inequality.


\section{Preliminaries} \label{preliminaries}

Throughout this paper, we let $b>0$ be fixed and let $\Psi$ be a smooth monotonic function on $[0,b]$ with $\Psi'>0$ almost everywhere. The assumption of smoothness is not always required, but it is included here so as to ensure that $\Psi$-fractional derivatives to all orders can be defined on appropriate function space domains.

\begin{definition}[Some function spaces associated with $\Psi$-operators \cite{Fahad,Sousa}]
Given $\Psi$ as above and $\alpha\in(0,1)$, the function space $C_{\alpha;\,\Psi}[0,b]$ is defined by
\[
C_{\alpha;\,\Psi} \left([0, b], \R \right)=\left\lbrace h: (0, b]\to \R\; : \; \big( \Psi(\cdot)-\Psi(0)\big)^{\alpha}h(\cdot)\in C[0, b]\right\rbrace,
\]
endowed with the norm 
\[
\left\|h\right\|_{C_{\alpha;\,\Psi}[0, b]}=\sup_{t\in[a, b]}\Big| \big(\Psi(t)-\Psi(0)\big)^{\alpha}h(t)\Big|.
\]
Furthermore, we define the function space $AC_{\Psi}^n[0, b]$ by
\[
AC_{\Psi}^n[0, b]=\left\lbrace h: [0, b]\to \R\; : \; \left(\frac{1}{\Psi'(t)}\cdot\frac{\mathrm{d}}{\mathrm{d}t}\right)^n\big(h(t)\big)\in AC[0, b]\right\rbrace.
\]
\end{definition}

\subsection{Fractional Integrals and Derivatives}
\begin{definition}[Riemann--Liouville and Caputo operators \cite{Kilbas}] \label{Def:RLC}
If $y$ is an absolutely integrable function defined on $[0, b]$, then the Riemann--Liouville (RL) fractional integral of order $\alpha\in(0, \infty)$ of the function $y$ is given by 
\begin{equation*}
\prescript{}{0}{\mathscr{I}}_{t}^{\alpha}y\left(t\right)=\frac{1}{\Gamma \left( \alpha\right)}\int_{0}^{t}\left(t-s\right)^{\alpha -1}y\left( s\right) \,\mathrm{d}s.
\end{equation*}
Further, if $y\in AC^n[0, b]$, then the Riemann--Liouville (RL) fractional derivative of order $\alpha\in(n-1,n)$ of the function $y$ is given by
\begin{equation*}
\prescript{RL}{0}{\mathscr{D}}^{\alpha}_{t}y(t)=\left(\frac{\mathrm{d}}{\mathrm{d}t}\right)^n\,\prescript{}{0}{\mathscr{I}}_{t}^{n-\alpha} y(t),
\end{equation*}
and the Caputo fractional derivative of order $\alpha\in(n-1,n)$ of the function $y$ is given by
\begin{equation*}
\prescript{C}{0}{\mathscr{D}}^{\alpha}_{t}y(t)=\prescript{}{0}{\mathscr{I}}_{t}^{n-\alpha}\,\left(\frac{\mathrm{d}}{\mathrm{d}t}\right)^n y(t).
\end{equation*}
Note that both Riemann--Liouville and Caputo derivatives are defined by compositions of the fractional integral with a finite repetition of the ordinary (first-order) derivative $\frac{\mathrm{d}}{\mathrm{d}t}$.
\end{definition}

\begin{definition}[Tempered fractional calculus \cite{Li,Fernandez}] \label{Def:tempered}
If $y$ is an absolutely integrable function defined on $[0, b]$, then the tempered fractional integral of order $\alpha\in(0, \infty)$ and index $\lambda\in\mathbb{R}$ of the function $y$ is given by 
\begin{equation*}
\prescript{T}{0}{\mathscr{I}}_{t}^{\alpha,\lambda}y\left(t\right)=\frac{1}{\Gamma \left( \alpha\right)}\int_{0}^{t}\left(t-s\right)^{\alpha -1}e^{-\lambda(t-s)}y\left( s\right) \,\mathrm{d}s.
\end{equation*}
Further, if $y\in AC^n[0, b]$, then the tempered fractional derivative of Riemann--Liouville and Caputo types, with order $\alpha\in(n-1,n)$ and index $\lambda\in\mathbb{R}$, of the function $y$ are given respectively by:
\begin{align*}
\prescript{TR}{0}{\mathscr{D}}^{\alpha,\lambda}_{t}y(t)&=\left(\frac{\mathrm{d}}{\mathrm{d}t}+\lambda\right)^n\,\prescript{T}{0}{\mathscr{I}}_{t}^{n-\alpha,\lambda} y(t), \\
\prescript{TC}{0}{\mathscr{D}}^{\alpha,\lambda}_{t}y(t)&=\prescript{T}{0}{\mathscr{I}}_{t}^{n-\alpha,\lambda}\,\left(\frac{\mathrm{d}}{\mathrm{d}t}+\lambda\right)^n y(t).
\end{align*}
Note that the tempered fractional derivatives of both Riemann--Liouville and Caputo type are defined by compositions of the tempered fractional integral with a finite repetition of a tempered variant of the ordinary (first-order) derivative, namely $\frac{\mathrm{d}}{\mathrm{d}t}+\lambda$. We can write the $n^{th}$-order tempered fractional derivative with index $\lambda\in\mathbb{R}$ of an $n$ times differentiable function $y$ as
\[
\prescript{T}{}{\mathscr{D}}^{n, \lambda} y(t)=\left(\frac{\mathrm{d}}{\mathrm{d}t}+\lambda\right)^n y(t).
\]
\end{definition}

\begin{definition}[$\Psi$-fractional calculus \cite{Kilbas,Osler,Almeida}] \label{Def:Psi}
If $y$ is an absolutely $\Psi$-integrable function defined on $[0, b]$ (i.e. assuming $y\in L^1\left([0,b],\mathrm{d}\Psi\right)$, then the $\Psi$-Riemann--Liouville fractional integral of order $\alpha\in(0, \infty)$ of the function $y$ is given by 
\begin{equation*}
\prescript{}{0}{\mathscr{I}}_{\Psi(t)}^{\alpha}y\left( t\right) =\frac{1}{\Gamma \left( \alpha\right) }\int_{0}^{t}\Psi ^{\prime }\left( s\right) \left( \Psi \left(t\right) -\Psi \left( s\right) \right) ^{\alpha -1}y\left( s\right) \,\mathrm{d}s.
\end{equation*}
Further, if $y$ is in the space $AC_{\Psi}^n[0, b]$ defined in \cite{Fahad}, then the $\Psi$-Riemann--Liouville and $\Psi$-Caputo fractional derivatives of order $\alpha\in(n-1,n)$ of the function $y$ are given respectively by:
\begin{align*}
\prescript{RL}{0}{\mathscr{D}}^{\alpha}_{\Psi(t)}y(t)= \left(\frac{1}{{\Psi}^{'}(t)}\cdot\frac{\mathrm{d}}{\mathrm{d}t}\right)^n\,\prescript{}{0}{\mathscr{I}}_{\Psi(t)}^{n-\alpha} y(t), \\
\prescript{C}{0}{\mathscr{D}}^{\alpha}_{\Psi(t)}y(t)=\prescript{}{0}{\mathscr{I}}_{\Psi(t)}^{n-\alpha}\,\left(\frac{1}{{\Psi}^{'}(t)}\cdot\frac{\mathrm{d}}{\mathrm{d}t}\right)^n y(t).
\end{align*}
Note that both $\Psi$-fractional derivatives (of Riemann--Liouville and Caputo type) are defined by compositions of the $\Psi$-fractional integral with a finite repetition of the ordinary (first-order) derivative with respect to $\Psi$, namely $\frac{1}{{\Psi}^{'}(t)}\cdot\frac{\mathrm{d}}{\mathrm{d}t}$.
\end{definition}

\begin{definition}[$\Psi$-tempered fractional calculus \cite{Fahad}] \label{Def:Psitempered}
If $y$ is an absolutely $\Psi$-integrable function defined on $[0, b]$, then the $\Psi$-tempered fractional integral of order $\alpha\in(0, \infty)$ and index $\lambda\in\mathbb{R}$ of the function $y$ is given by
\begin{equation*}
\prescript{T}{0}{\mathscr{I}}_{\Psi(t)}^{\alpha,\lambda}y\left(t\right)=\frac{1}{\Gamma \left( \alpha\right)}\int_{0}^{t}\Psi'(s)\left(\Psi(t)-\Psi(s)\right)^{\alpha -1}e^{-\lambda\left(\Psi(t)-\Psi(s)\right)}\,y\left( s\right) \,\mathrm{d}s.
\end{equation*}
Further, if $y$ is in $AC_{\Psi}^n[0, b]$, then the $\Psi$-tempered fractional derivatives of Riemann--Liouville and Caputo types, with order $\alpha\in(n-1,n)$ and index $\lambda\in\mathbb{R}$, of the function $y$ are given respectively by:
\begin{align*}
\prescript{TR}{0}{\mathscr{D}}^{\alpha,\lambda}_{\Psi(t)}y(t)=\left(\frac{1}{{\Psi}^{'}(t)}\cdot\frac{\mathrm{d}}{\mathrm{d}t}+\lambda\right)^n\,\prescript{T}{0}{\mathscr{I}}_{\Psi(t)}^{n-\alpha,\lambda} y(t), \\
\prescript{TC}{0}{\mathscr{D}}^{\alpha,\lambda}_{\Psi(t)}y(t)=\prescript{T}{0}{\mathscr{I}}_{\Psi(t)}^{n-\alpha,\lambda}\,\left(\frac{1}{{\Psi}^{'}(t)}\cdot\frac{\mathrm{d}}{\mathrm{d}t}+\lambda\right)^n y(t).
\end{align*}
Note that both $\Psi$-tempered derivatives (of Riemann--Liouville and Caputo type) are defined by compositions of the $\Psi$-tempered integral with a finite repetition of the ordinary (first-order) tempered derivative with respect to $\Psi$, namely $\prescript{T}{}{\mathscr{D}}^{1,\lambda}_{\Psi(t)}=\frac{1}{{\Psi}^{'}(t)}\cdot\frac{\mathrm{d}}{\mathrm{d}t}+\lambda$ giving the iterated version $\prescript{T}{}{\mathscr{D}}^{n,\lambda}_{\Psi(t)}=\left(\frac{1}{{\Psi}^{'}(t)}\cdot\frac{\mathrm{d}}{\mathrm{d}t}+\lambda\right)^n$, for $n\in\N$.
\end{definition}

\begin{rem}
As was observed in \cite{Fahad}, the tempered fractional derivative with respect to $\Psi(t)$ is identical to the Hadamard-type fractional derivative with respect to $e^{\Psi(t)}$. Thus, the theory of tempered fractional calculus with respect to functions is exactly the same as the theory of Hadamard-type fractional calculus with respect to functions. As the previous work \cite{Fahad} focused on the theory from the viewpoint of Hadamard-type fractional calculus, using the notation of Hadamard-type operators with respect to a function, in this paper we shall consider the same theory from the alternative viewpoint of tempered fractional calculus.
\end{rem}

\begin{proposition}[\cite{Fahad}]
The following relations between the above definitions are clear:
\begin{itemize}
\item If $\lambda=0$, then the tempered fractional operators (Definition \ref{Def:tempered}) become the classical fractional operators (Definition \ref{Def:RLC}).
\item If $\Psi(t)=t$, then the $\Psi$-fractional operators (Definition \ref{Def:Psi}) become the classical fractional operators (Definition \ref{Def:RLC}).
\item If $\lambda=0$, then the $\Psi$-tempered fractional operators (Definition \ref{Def:Psitempered}) become the $\Psi$-fractional operators (Definition \ref{Def:Psi}).
\item If $\Psi(t)=t$, then the $\Psi$-tempered fractional operators (Definition \ref{Def:Psitempered}) become the tempered fractional operators (Definition \ref{Def:tempered}).
\item If $\lambda=0$ and $\Psi(t)=t$, then the $\Psi$-tempered fractional operators (Definition \ref{Def:Psitempered}) become the classical fractional operators (Definition \ref{Def:RLC}).
\end{itemize}
\end{proposition}

Note that the tempered fractional calculus (Definition \ref{Def:tempered}) is a special case of fractional calculus with analytic kernels \cite{fernandez-ozarslan-baleanu} and also a special case of weighted fractional calculus \cite{ff:weighted} -- in fact, it is the unique intersection of these two general classes. This means that $\Psi$-tempered fractional calculus (Definition \ref{Def:Psitempered}) is a special case of fractional calculus with analytic kernels with respect to functions \cite{Oumarou} and also a special case of weighted fractional calculus with respect to functions \cite{agrawal}. As a result, many properties of $\Psi$-tempered fractional calculus can be deduced directly from the more general results found in the aforementioned classes of operators \cite{agrawal,ff:weighted,Oumarou,jarad-abdeljawad-shah}.

For example, \cite[Theorems 3.5 and 3.14]{Oumarou}, together with \cite[Theorem 2.1]{Fernandez} for tempered fractional calculus, gives the following series formulae for the $\Psi$-tempered operators:
\begin{align}
\prescript{T}{0}{\mathscr{I}}_{\Psi(t)}^{\alpha, \lambda}\,y(t)&=\sum_{k=0}^{\infty}\frac{(-\lambda)^k\,\Gamma(\alpha+k)}{k!\,\Gamma(\alpha) }\prescript{}{0}{\mathscr{I}}_{\Psi(t)}^{\alpha+k}\,y(t), \label{series:I} \\
\prescript{TR}{0}{\mathscr{D}}_{\Psi(t)}^{\alpha, \lambda}\,y(t)&=\sum_{k=0}^{\infty}\frac{(-\lambda)^k\,\Gamma(-\alpha+k)}{k!\,\Gamma(-\alpha) }\,\, \prescript{RL}{0}{\mathscr{D}}_{\Psi(t)}^{\alpha}\,\prescript{}{0}{\mathscr{I}}_{\Psi(t)}^{k}\,y(t), \label{series:D}
\end{align}
where $\lambda\in\mathbb{R}$ and (in the first case) $\alpha\in(0,\infty)$ and $y$ is absolutely $\Psi$-integrable on $[0,b]$, or (in the second case) $n-1\leq\alpha<n\in\mathbb{N}$ and $y$ is in $AC_{\Psi}^n[0, b]$.

\subsection{Conjugation relations and their consequences}

\begin{proposition}[Conjugation relations \cite{Kilbas,Fernandez,Fahad}] \label{Conjugationrelations}
The $\Psi$-fractional operators (Definition \ref{Def:Psi}) can be written \cite{Kilbas} as conjugations of the classical fractional operators (Definition \ref{Def:RLC}) with the operator $Q_{\Psi}$ defined by $Q_{\Psi}y=y\circ\Psi$:
\begin{align*}
\prescript{}{0}{\mathscr{I}}_{\Psi(t)}^{\alpha}&=Q_{\Psi}\circ\prescript{}{\Psi(0)}{\mathscr{I}}_{t}^{\alpha}\circ Q_{\Psi}^{-1}, \\
\prescript{RL}{0}{\mathscr{D}}_{\Psi(t)}^{\alpha}&=Q_{\Psi}\circ\prescript{RL}{\Psi(0)}{\mathscr{D}}_{t}^{\alpha}\circ Q_{\Psi}^{-1}, \\
\prescript{C}{0}{\mathscr{D}}_{\Psi(t)}^{\alpha}&=Q_{\Psi}\circ\prescript{C}{\Psi(0)}{\mathscr{D}}_{t}^{\alpha}\circ Q_{\Psi}^{-1},
\end{align*}
where $\circ$ denotes composition of functions or operators.

The tempered fractional operators (Definition \ref{Def:tempered}) can be written \cite{Fernandez} as conjugations of the classical fractional operators (Definition \ref{Def:RLC}) with the operator $M_{\exp(\lambda t)}$ defined by $\left(M_{\exp(\lambda t)}y\right)(t)=e^{\lambda t}y(t)$:
\begin{align*}
\prescript{T}{0}{\mathscr{I}}_{t}^{\alpha,\lambda}=M_{\exp(\lambda t)}^{-1}\circ \prescript{}{0}{\mathscr{I}}_{t}^{\alpha}\circ M_{\exp(\lambda t)}\qquad&\text{ i.e. }\qquad\prescript{T}{0}{\mathscr{I}}_{t}^{\alpha,\lambda}y(t)=e^{-\lambda t}\prescript{}{0}{\mathscr{I}}_{t}^{\alpha}\left(e^{\lambda t}y(t)\right), \\
\prescript{TR}{0}{\mathscr{D}}_{t}^{\alpha,\lambda}=M_{\exp(\lambda t)}^{-1}\circ \prescript{RL}{0}{\mathscr{D}}_{t}^{\alpha}\circ M_{\exp(\lambda t)}\qquad&\text{ i.e. }\qquad\prescript{TR}{0}{\mathscr{D}}_{t}^{\alpha,\lambda}y(t)=e^{-\lambda t}\prescript{RL}{0}{\mathscr{D}}_{t}^{\alpha}\left(e^{\lambda t}y(t)\right), \\
\prescript{TC}{0}{\mathscr{D}}_{t}^{\alpha,\lambda}=M_{\exp(\lambda t)}^{-1}\circ \prescript{C}{0}{\mathscr{D}}_{t}^{\alpha}\circ M_{\exp(\lambda t)}\qquad&\text{ i.e. }\qquad\prescript{TC}{0}{\mathscr{D}}_{t}^{\alpha,\lambda}y(t)=e^{-\lambda t}\prescript{C}{0}{\mathscr{D}}_{t}^{\alpha}\left(e^{\lambda t}y(t)\right),
\end{align*}

The $\Psi$-tempered fractional operators (Definition \ref{Def:Psitempered}) can be written \cite{Fahad} as conjugations of the classical fractional operators (Definition \ref{Def:RLC}) as follows:
\begin{align*}
\prescript{T}{0}{\mathscr{I}}_{\Psi(t)}^{\alpha,\lambda}&=Q_{\Psi}\circ M_{\exp(\lambda t)}^{-1}\circ \prescript{}{\Psi(0)}{\mathscr{I}}_{t}^{\alpha}\circ M_{\exp(\lambda t)}\circ Q_{\Psi}^{-1} \\ &=M_{\exp(\lambda\Psi(t))}^{-1}\circ Q_{\Psi}\circ\prescript{}{\Psi(0)}{\mathscr{I}}_{t}^{\alpha}\circ Q_{\Psi}^{-1}\circ M_{\exp(\lambda t)}, \\
\prescript{TR}{0}{\mathscr{D}}_{\Psi(t)}^{\alpha,\lambda}&=Q_{\Psi}\circ M_{\exp(\lambda t)}^{-1}\circ \prescript{RL}{\Psi(0)}{\mathscr{D}}_{t}^{\alpha}\circ M_{\exp(\lambda t)}\circ Q_{\Psi}^{-1} \\ &=M_{\exp(\lambda\Psi(t))}^{-1}\circ Q_{\Psi}\circ\prescript{RL}{\Psi(0)}{\mathscr{D}}_{t}^{\alpha}\circ Q_{\Psi}^{-1}\circ M_{\exp(\lambda t)}, \\
\prescript{TC}{0}{\mathscr{D}}_{\Psi(t)}^{\alpha,\lambda}&=Q_{\Psi}\circ M_{\exp(\lambda t)}^{-1}\circ \prescript{C}{\Psi(0)}{\mathscr{D}}_{t}^{\alpha}\circ M_{\exp(\lambda t)}\circ Q_{\Psi}^{-1} \\ &=M_{\exp(\lambda\Psi(t))}^{-1}\circ Q_{\Psi}\circ\prescript{C}{\Psi(0)}{\mathscr{D}}_{t}^{\alpha}\circ Q_{\Psi}^{-1}\circ M_{\exp(\lambda t)}.
\end{align*}
In particular, the second operator identity in each case can be rewritten in the following way:
\begin{align}
\prescript{T}{0}{\mathscr{I}}_{\Psi(t)}^{\alpha,\lambda}y(t)&=e^{-\lambda\Psi(t)}\prescript{}{0}{\mathscr{I}}_{\Psi(t)}^{\alpha}\left(e^{\lambda\Psi(t)}y(t)\right), \label{Tprod:I} \\
\prescript{TR}{0}{\mathscr{D}}_{\Psi(t)}^{\alpha,\lambda}y(t)&=e^{-\lambda\Psi(t)}\prescript{RL}{0}{\mathscr{D}}_{\Psi(t)}^{\alpha}\left(e^{\lambda\Psi(t)}y(t)\right), \label{Tprod:R} \\
\prescript{TC}{0}{\mathscr{D}}_{\Psi(t)}^{\alpha,\lambda}y(t)&=e^{-\lambda\Psi(t)}\prescript{C}{0}{\mathscr{D}}_{\Psi(t)}^{\alpha}\left(e^{\lambda\Psi(t)}y(t)\right), \label{Tprod:C}
\end{align}
giving a simple relationship between $\Psi$-tempered fractional calculus and the original tempered fractional calculus.
\end{proposition}

Conjugation relations as listed above are immensely important in studying and understanding operators \cite{ff:conjug} such as those of tempered fractional calculus and $\Psi$-fractional calculus, and therefore by extension also $\Psi$-tempered fractional calculus. For example, the following semigroup property of $\Psi$-tempered fractional integrals is now an immediate consequence of the corresponding (well-known) property of Riemann--Liouville fractional integrals \cite[Lemma 2.3]{Kilbas}:
\begin{equation}
\label{semi:I}
\prescript{T}{0}{\mathscr{I}}_{\Psi(t)}^{\alpha_1,\lambda}\,\prescript{T}{0}{\mathscr{I}}_{\Psi(t)}^{\alpha_2, \lambda}\,y(t)=\prescript{T}{0}{\mathscr{I}}_{\Psi(t)}^{\alpha_1+\alpha_2, \lambda}\,y(t),\qquad\alpha_1,\alpha_2\in(0,\infty),\,\lambda\in\mathbb{R},
\end{equation}
this identity being true for $\Psi$-a.e. $t\in[0,b]$, or for all $t\in[0,b]$ if $\alpha_1+\alpha_2\geq1$, and where $y$ is an absolutely $\Psi$-integrable function on $[0,b]$. Similarly, from a well-known property of Riemann--Liouville fractional derivatives \cite[Property 2.3]{Kilbas}, we obtain immediately the following semigroup property of $\Psi$-tempered derivatives of Riemann--Liouville type:
\begin{multline*}
\prescript{T}{}{\mathscr{D}}_{\Psi(t)}^{m,\lambda}\,\prescript{TR}{0}{\mathscr{D}}_{\Psi(t)}^{\alpha,\lambda}\,y(t)=\left(\frac{1}{{\Psi}^{'}(t)}\cdot\frac{\mathrm{d}}{\mathrm{d}t}+\lambda\right)^m\prescript{TR}{0}{\mathscr{D}}_{\Psi(t)}^{\alpha,\lambda}\,y(t) \\ =\prescript{TR}{0}{\mathscr{D}}_{\Psi(t)}^{\alpha+m,\lambda}\,y(t),\qquad\alpha>0,\,m\in\mathbb{N},\,\lambda\in\mathbb{R},
\end{multline*}
this identity being true for any function $y$ such that both expressions exist. (In the derivative of integer order $m$, we have replaced the ``TR'' notation by simply ``T'' and suppressed the ``$0$'', as in this case there is no initial value dependence and no difference between Riemann--Liouville and Caputo type derivatives.) Finally, another well-known property of Riemann--Liouville fractional derivatives and integrals \cite[Theorem 2.5]{Samko} yields immediately the following semigroup property of $\Psi$-tempered operators:
\begin{equation} \label{DI}
\prescript{TR}{0}{\mathscr{D}}_{\Psi(t)}^{\alpha,\lambda}\,\left[\prescript{T}{0}{\mathscr{I}}_{\Psi(t)}^{\beta,\lambda}\, y(t)\right] =
\begin{cases}
\prescript{T}{0}{\mathscr{I}}_{\Psi(t)}^{\beta-\alpha,\lambda}\,y(t), &\quad \text{if }\alpha<\beta; \\
\qquad y(t), &\quad\text{if }\alpha=\beta; \\
\prescript{TR}{0}{\mathscr{D}}_{\Psi(t)}^{\alpha-\beta,\lambda}\,y(t), &\quad \text{if }\alpha>\beta,
\end{cases}
\end{equation}
where in all cases $\alpha,\beta\in(0,\infty)$ and $\lambda\in\mathbb{R}$ is arbitrary, and where $y$ is absolutely $\Psi$-integrable in the first two cases, or any function such that both expressions exist in the last case.

Note that all three of the above semigroup relations were already proved in \cite[Proposition 3.12]{Fahad}, so we do not formalise them into a Proposition or Theorem here. We have reproduced them just in order to show the results more easily for the reader, in a way so that the different possible cases are more readily understandable without needing the complexity of analytic continuation in the fractional order.

\section{Further properties of $\Psi$-tempered fractional calculus} \label{Sec:properties}


\subsection{Limiting behaviour}

It is clear from the definition of the $\Psi$-tempered fractional integral that it is equal to $0$ at the lower limit point $t\to0$ if the function $y$ is bounded on $[0,b]$:
\begin{align*}
\left|\prescript{T}{0}{\mathscr{I}}_{\Psi(t)}^{\alpha,\lambda}y\left(t\right)\right|&\leq\frac{1}{\Gamma \left( \alpha\right)}\int_{0}^{t}\Psi'(s)\left(\Psi(t)-\Psi(s)\right)^{\alpha -1}e^{-\lambda\left(\Psi(t)-\Psi(s)\right)}\big|y\left( s\right)\big| \,\mathrm{d}s \\
&\leq\frac{\sup_{[0,b]}|y|}{\Gamma \left( \alpha\right)}\int_{0}^{t}\Psi'(s)\left(\Psi(t)-\Psi(s)\right)^{\alpha -1}\,\mathrm{d}s \\
&=\frac{\sup_{[0,b]}|y|}{\Gamma \left( \alpha+1\right)}\left(\Psi(t)-\Psi(0)\right)^{\alpha},
\end{align*}
and this upper bound tends to $0$ as $t\to0$, provided that $\alpha>0$. Therefore, we have the following conclusion:
\begin{equation}
\label{limit0:int}
\lim_{t\to0^+}\left[\prescript{T}{0}{\mathscr{I}}_{\Psi(t)}^{\alpha,\lambda}y\left(t\right)\right]=0,\qquad y\in C[0,b].
\end{equation}
It follows then, from the definition of the $\Psi$-tempered derivative of Caputo type, that:
\begin{equation}
\label{limit0:Cap}
\lim_{t\to0^+}\left[\prescript{TC}{0}{\mathscr{D}}_{\Psi(t)}^{\alpha,\lambda}y\left(t\right)\right]=0,\qquad y\in C^n[0,b],\quad n-1<\alpha<n.
\end{equation}

Having investigated the limiting behaviour of some $\Psi$-tempered fractional operators at the left limit point of the function's domain, let us also examine their limiting behaviour at integer limit points within the domain for the fractional order of differentiation. It is already known \cite[Proposition 3.11]{Fahad} that the $\Psi$-tempered Riemann--Liouville integral and derivative of order $\alpha$ and index $\lambda$ are continuous and analytic functions of $\alpha$, so all of their possible limiting behaviours are well understood. For the Caputo-type derivative, the result is given as follows.

\begin{theorem}\label{t21}
If $\lambda\in\mathbb{R}$ and $n\in\N$, then
\[
\lim_{\alpha\to n^-}\left[\prescript{TC}{0}{\mathscr{D}}_{\Psi(t)}^{\alpha, \lambda} y(t)\right]=\prescript{T}{}{\mathscr{D}}_{\Psi(t)}^{n, \lambda} y(t)
\]
and
\[
\lim_{\alpha\to(n-1)^+}\left[\prescript{TC}{0}{\mathscr{D}}_{\Psi(t)}^{\alpha, \lambda} y(t)\right]=\prescript{T}{}{\mathscr{D}}_{\Psi(t)}^{n-1, \lambda} y(t)-e^{-\lambda(\Psi(t)-\Psi(0))}\left[\prescript{T}{}{\mathscr{D}}_{\Psi(t)}^{n-1,\lambda}\,y(t)\right]_{t=0}.
\]
\end{theorem}

\begin{proof}
Taking $n-1<\alpha<n$ without loss of generality, we have
\begin{equation}
\label{limiting:step}
\prescript{TC}{0}{\mathscr{D}}_{\Psi(t)}^{\alpha, \lambda}\,y(t)=\frac{e^{-\lambda\Psi(t)}}{\Gamma(n-\alpha)}\int_{0}^{t} \Psi'(s)\left(\Psi(t)-\Psi(s)\right)^{n-\alpha-1} \left( \frac{1}{\Psi'(s)}\cdot\frac{\mathrm{d}}{\mathrm{d}s}\right)^{n}\left(e^{\lambda\Psi(s)} y(s)\right)\,\mathrm{d}s.
\end{equation}
Taking the limit as $\alpha\to(n-1)^+$ on both sides of \eqref{limiting:step} gives
\begin{align*}
\lim_{\alpha\to(n-1)^+}\left[\prescript{TC}{0}{\mathscr{D}}_{\Psi(t)}^{\alpha, \lambda}\,y(t)\right]&=e^{-\lambda\Psi(t)}\int_{0}^{t} \Psi'(s)\left( \frac{1}{\Psi'(s)}\cdot\frac{\mathrm{d}}{\mathrm{d}s}\right)^{n}\left(e^{\lambda\Psi(s)} y(s)\right)\,\mathrm{d}s \\
&\hspace{-2cm}=e^{-\lambda\Psi(t)}\int_{0}^{t} \frac{\mathrm{d}}{\mathrm{d}s}\left( \frac{1}{\Psi'(s)}\cdot\frac{\mathrm{d}}{\mathrm{d}s}\right)^{n-1}\left(e^{\lambda\Psi(s)} y(s)\right)\,\mathrm{d}s \\
&\hspace{-2cm}=e^{-\lambda\Psi(t)}\Bigg[\left( \frac{1}{\Psi'(t)}\cdot\frac{\mathrm{d}}{\mathrm{d}t}\right)^{n-1}\left(e^{\lambda\Psi(t)} y(t)\right) \\ &-\left[\left( \frac{1}{\Psi'(t)}\cdot\frac{\mathrm{d}}{\mathrm{d}t}\right)^{n-1}\left(e^{\lambda\Psi(t)} y(t)\right)\right]_{t=0}\Bigg] \\
&\hspace{-2cm}=\prescript{T}{}{\mathscr{D}}_{\Psi(t)}^{n-1,\lambda}-e^{-\lambda(\Psi(t)-\Psi(0))}\left[e^{-\lambda\Psi(t)}\left( \frac{1}{\Psi'(t)}\cdot\frac{\mathrm{d}}{\mathrm{d}t}\right)^{n-1}\left(e^{\lambda\Psi(t)} y(t)\right)\right]_{t=0} \\
&\hspace{-2cm}=\prescript{T}{}{\mathscr{D}}_{\Psi(t)}^{n-1,\lambda}-e^{-\lambda(\Psi(t)-\Psi(0))}\left[\prescript{T}{}{\mathscr{D}}_{\Psi(t)}^{n-1,\lambda}\,y(t)\right]_{t=0},
\end{align*}
which is the second of the two required results. To get the first one, we integrate \eqref{limiting:step} by parts to get
\begin{multline*}
\prescript{TC}{0}{\mathscr{D}}_{\Psi(t)}^{\alpha, \lambda}\,y(t)=e^{-\lambda\Psi(t)}\Bigg\lbrace \frac{(\Psi(t)-\Psi(0))^{n-\alpha}}{\Gamma(n-\alpha+1)} \left[  \left( \frac{1}{\Psi'(t)}\cdot\frac{\mathrm{d}}{\mathrm{d}t}\right)^{n} \left( e^{\lambda\Psi(t)} y(t)\right) \right]_{t=0} \\
+ \frac{1}{\Gamma(n-\alpha+1)} \int_{0}^{t}\left(\Psi(t)-\Psi(s)\right)^{n-\alpha}\,\frac{\mathrm{d}}{\mathrm{d}s}\left[  \left( \frac{1}{\Psi'(s)}\cdot\frac{\mathrm{d}}{\mathrm{d}s}\right)^{n} \left( e^{\lambda\Psi(s)} y(s)\right) \right] \,\mathrm{d}s\Bigg\rbrace.
\end{multline*}
Taking the limit as $\alpha\to n^- $ on both sides, we get
\begin{align*}
\lim_{\alpha\to n^-}\left[\prescript{TC}{0}{\mathscr{D}}_{\Psi(t)}^{\alpha, \lambda} y(t)\right]
&=e^{-\lambda\Psi(t)}\Bigg\lbrace \left[ \left( \frac{1}{\Psi'(t)}\cdot\frac{\mathrm{d}}{\mathrm{d}t}\right)^{n} \left( e^{\lambda\Psi(t)} y(t)\right) \right]_{t=0} \\
&\hspace{2cm}+ \int_{0}^{t}\frac{\mathrm{d}}{\mathrm{d}s}\left[  \left( \frac{1}{\Psi'(s)}\cdot\frac{\mathrm{d}}{\mathrm{d}s}\right)^{n} \left( e^{\lambda\Psi(s)} y(s)\right) \right] \,\mathrm{d}s\Bigg\rbrace \\
&\hspace{-2cm}=e^{-\lambda\Psi(t)}\Bigg\lbrace \left[ \left( \frac{1}{\Psi'(t)}\cdot\frac{\mathrm{d}}{\mathrm{d}t}\right)^{n} \left( e^{\lambda\Psi(t)} y(t)\right) \right]_{t=0} \\
&+\left( \frac{1}{\Psi'(t)}\cdot\frac{\mathrm{d}}{\mathrm{d}t}\right)^{n} \left( e^{\lambda\Psi(t)} y(t)\right)-\left[ \left( \frac{1}{\Psi'(t)}\cdot\frac{\mathrm{d}}{\mathrm{d}t}\right)^{n} \left( e^{\lambda\Psi(t)} y(t)\right) \right]_{t=0}\Bigg\rbrace \\
&\hspace{-2cm}= e^{-\lambda\Psi(t)} \left( \frac{1}{\Psi'(t)}\cdot \frac{\mathrm{d}}{\mathrm{d}t}\right)^{n} \left( e^{\lambda\Psi(t)} y(t)\right)\\
&\hspace{-2cm}=\prescript{T}{}{\mathscr{D}}_{\Psi(t)}^{n, \lambda}\, y(t).
\end{align*} 
\end{proof}

\subsection{Example functions}

It is well known that the Riemann--Liouville operators act on the set of functions of the form $\frac{t^{\xi}}{\Gamma(\xi+1)}$ by increasing or decreasing the index $\xi$ by the order of integration or differentiation, and it follows from conjugation relations \cite[Properties 2.18 and 2.20]{Kilbas} that the $\Psi$-fractional operators act similarly on the set of functions of the form $\frac{(\Psi(t)-\Psi(0))^{\xi}}{\Gamma(\xi+1)}$. Similarly, the following action of the $\Psi$-tempered fractional operators follows immediately from conjugation relations too.

\begin{proposition}[{\cite[Proposition 3.13]{Fahad}}] \label{Prop:specfunc1}
If $\alpha\in(0,\infty)$, $\lambda\in\mathbb{R}$ and $\xi\in(-1,\infty)$, then
\begin{align*}
\prescript{T}{0}{\mathscr{I}}_{\Psi(t)}^{\alpha, \lambda} \left( e^{-\lambda\Psi(t)}\left( \Psi(t)-\Psi(0)\right)^\xi\right) &=\frac{\Gamma(\xi+1)}{\Gamma(\xi+\alpha+1)}\,e^{-\lambda\Psi(t)}\left( \Psi(t)-\Psi(0)\right)^{\xi+\alpha}, \\
\prescript{TR}{0}{\mathscr{D}}_{\Psi(t)}^{\alpha, \lambda} \left( e^{-\lambda\Psi(t)}\left( \Psi(t)-\Psi(0)\right)^\xi\right) &=\frac{\Gamma(\xi+1)}{\Gamma(\xi-\alpha+1)}\,e^{-\lambda\Psi(t)}\left( \Psi(t)-\Psi(0)\right)^{\xi-\alpha},
\end{align*}
and if in addition $\xi>\lfloor\alpha\rfloor$, then
\[
\prescript{TC}{0}{\mathscr{D}}_{\Psi(t)}^{\alpha, \lambda} \left( e^{-\lambda\Psi(t)}\left( \Psi(t)-\Psi(0)\right)^\xi\right)=\frac{\Gamma(\xi+1)}{\Gamma(\xi-\alpha+1)}\,e^{-\lambda\Psi(t)}\left( \Psi(t)-\Psi(0)\right)^{\xi-\alpha}.
\]
\end{proposition}

The application of the $\Psi$-tempered operators to powers of $\Psi(t)-\Psi(0)$ can be given as follows.

\begin{proposition}\label{h1}
If $\alpha\in (0, \infty)$, $\lambda\in\mathbb{R}$ and $\xi\in (-1, \infty)$, then
\begin{multline*}
\prescript{T}{0}{\mathscr{I}}_{\Psi(t)}^{\alpha, \lambda}\left(\Psi(t)-\Psi(0) \right)^\xi= \frac{\Gamma(\xi+1)}{\Gamma(\xi+\alpha+1)}\left(\Psi(t)-\Psi(0) \right)^{\xi+\alpha}\,e^{-\lambda\left(\Psi(t)-\Psi(0) \right)} \\ \times\prescript{}{1}F_1\Big(\xi+1; \xi+\alpha+1; \lambda\left(\Psi(t)-\Psi(0) \right)\Big),
\end{multline*}
and
\begin{multline*}
\prescript{TR}{0}{\mathscr{D}}_{\Psi(t)}^{\alpha, \lambda}\left(\Psi(t)-\Psi(0) \right)^\xi \\ = \frac{\Gamma(\xi+1)}{\Gamma(\xi-\alpha+1)}\left(\Psi(t)-\Psi(0) \right)^{\xi-\alpha}\,e^{-\lambda\left(\Psi(t)-\Psi(0) \right)} \,\prescript{}{1}F_1\Big(\xi+1; \xi-\alpha+1; \lambda\left(\Psi(t)-\Psi(0) \right)\Big),
\end{multline*}
and if in addition $\xi>\lfloor\alpha\rfloor$, then
\begin{multline*}
\prescript{TC}{0}{\mathscr{D}}_{\Psi(t)}^{\alpha, \lambda}\left(\Psi(t)-\Psi(0) \right)^\xi \\ = \frac{\Gamma(\xi+1)}{\Gamma(\xi-\alpha+1)}\left(\Psi(t)-\Psi(0) \right)^{\xi-\alpha}\,e^{-\lambda\left(\Psi(t)-\Psi(0) \right)} \,\prescript{}{1}F_1\Big(\xi+1; \xi-\alpha+1; \lambda\left(\Psi(t)-\Psi(0) \right)\Big),
\end{multline*}
where $M(a;c;z)=\prescript{}{1}F_1(a;c;z)$ is the well-known Kummer's function or confluent hypergeometric function \cite[Chapter XVI]{whittaker-watson}.
\end{proposition}

\begin{proof}
Using the expression for the $\Psi$-tempered integral as a conjugation of the $\Psi$-Riemann--Liouville integral, together with the action of the latter on functions of the form $\frac{(\Psi(t)-\Psi(0))^{\xi}}{\Gamma(\xi+1)}$, we have:
\begin{align*}
\prescript{T}{0}{\mathscr{I}}_{\Psi(t)}^{\alpha, \lambda}\left(\Psi(t)-\Psi(0) \right)^\xi
&= e^{-\lambda\left(\Psi(t)-\Psi(0) \right)}\prescript{}{0}{\mathscr{I}}_{\Psi(t)}^{\alpha}\left( e^{\lambda\left(\Psi(t)-\Psi(0) \right)}\left(\Psi(t)-\Psi(0) \right)^\xi\right) \\
&\hspace{-1cm}= e^{-\lambda\left(\Psi(t)-\Psi(0) \right)}\prescript{}{0}{\mathscr{I}}_{\Psi(t)}^{\alpha}\left( \sum_{k=0}^{\infty}\frac{(\lambda \left(\Psi(t)-\Psi(0) \right))^k}{k!}\left(\Psi(t)-\Psi(0) \right)^\xi\right) \\
&\hspace{-1cm}= e^{-\lambda\left(\Psi(t)-\Psi(0) \right)}\sum_{k=0}^{\infty}\, \frac{\lambda^k}{k!}\,\prescript{}{0}{\mathscr{I}}_{\Psi(t)}^{\alpha}\left(\Psi(t)-\Psi(0) \right)^{\xi+k} \\
&\hspace{-1cm}=e^{-\lambda\left(\Psi(t)-\Psi(0) \right)}\sum_{k=0}^{\infty}\, \frac{\lambda^k}{k!}\cdot\frac{\Gamma(\xi+k+1)}{\Gamma(\xi+k+\alpha+1)}\,\left(\Psi(t)-\Psi(0) \right)^{\xi+k+\alpha}\\
&\hspace{-1cm}=e^{-\lambda\left(\Psi(t)-\Psi(0) \right)} \left(\Psi(t)-\Psi(0) \right)^{\xi+\alpha} \frac{\Gamma(\xi+1)}{\Gamma(\xi+\alpha+1)}\\
&\times\left\lbrace \frac{\Gamma(\xi+\alpha+1)}{\Gamma(\xi+1)}\sum_{k=0}^{\infty}\,\,  \frac{\Gamma(\xi+1+k)}{\Gamma(\xi+\alpha+1+k)}\,\frac{\left( \lambda\left(\Psi(t)-\Psi(0) \right)\right) ^{k}}{k!}\right\rbrace \\
&\hspace{-1cm}=\frac{\Gamma(\xi+1)}{\Gamma(\xi+\alpha+1)}\left(\Psi(t)-\Psi(0) \right)^{\xi+\alpha} \,e^{-\lambda\left(\Psi(t)-\Psi(0) \right)}\\
&\hspace{2cm}\times\prescript{}{1}F_1\Big(\xi+1; \xi+\alpha+1; \lambda\left(\Psi(t)-\Psi(0) \right)\Big).
\end{align*}
This proves the first of the stated identities (for the $\Psi$-tempered integral). The identity for the $\Psi$-tempered derivative of Riemann--Liouville type can either be proved similarly by direct calculation, or deduced from the first one by analytic continuation \cite[Proposition 3.11]{Fahad}. For the $\Psi$-tempered derivative of Caputo type, we can repeat the above calculations with the fractional integral operator replaced by a Caputo-type derivative operator, requiring $\xi>n-1$ as usual for the Caputo-type derivative of a power function.
\end{proof}

\begin{proposition} \label{hypergeom:altern}
If $\alpha\in (0, \infty)$, $\lambda\in\mathbb{R}$ and $\xi\in (-1, \infty)$, then
\begin{multline*}
\prescript{T}{0}{\mathscr{I}}_{\Psi(t)}^{\alpha, \lambda}\left(\Psi(t)-\Psi(0) \right)^\xi=\frac{\Gamma(\xi+1)}{\Gamma(\xi+\alpha+1)} \left(\Psi(t)-\Psi(0) \right)^{\xi+\alpha} \\ \times\prescript{}{1}F_1\Big(\alpha; \xi+\alpha+1; -\lambda\left(\Psi(t)-\Psi(0) \right)\Big),
\end{multline*}
and
\begin{multline*}
\prescript{TR}{0}{\mathscr{D}}_{\Psi(t)}^{\alpha, \lambda}\left(\Psi(t)-\Psi(0) \right)^\xi=\frac{\Gamma(\xi+1)}{\Gamma(\xi-\alpha+1)} \left(\Psi(t)-\Psi(0) \right)^{\xi-\alpha} \\ \times\prescript{}{1}F_1\Big(-\alpha; \xi-\alpha+1; -\lambda\left(\Psi(t)-\Psi(0) \right)\Big).
\end{multline*}
\end{proposition}

\begin{proof}
This follows immediately from \cite[Example 2.1]{Fernandez} together with the $\Psi$-conjugation relation, or it can be proved directly by using the series formulae \eqref{series:I}--\eqref{series:D}.
\end{proof}

\begin{rem}
The results of Proposition \ref{h1} and Proposition \ref{hypergeom:altern} are equivalent to each other, by the well-known identity $\prescript{}{1}F_1(a;c;z) =e^z\,\prescript{}{1}F_1(c-a; c; -z)$ \cite[Eq. (7)]{Webb}.
\end{rem}

\begin{rem}
As special cases of Proposition \ref{h1} and Proposition \ref{hypergeom:altern}, we can recover already known results: on fractional calculus with respect to functions \cite{Kilbas,Almeida}, by setting $\lambda=0$; on tempered fractional calculus \cite{Fernandez}, by setting $\Psi(t)=t$; on Riemann--Liouville fractional calculus \cite{Samko,Kilbas}, by setting $\lambda=0$ and $\Psi(t)=t$.
\end{rem}

\begin{proposition} \label{Coroll:1}
If $\alpha\in(0, \infty)$ and $\lambda\in\mathbb{R}$, then:
\begin{align*}
\prescript{T}{0}{\mathscr{I}}_{\Psi(t)}^{\alpha, \lambda}(1)= e^{-\lambda(\Psi(t)-\Psi(0))}(\Psi(t)-\Psi(0))^\alpha E_{1,\, 1+\alpha}\Big(\lambda(\Psi(t)-\Psi(0)\Big), \\
\prescript{TR}{0}{\mathscr{D}}_{\Psi(t)}^{\alpha, \lambda}(1)= e^{-\lambda(\Psi(t)-\Psi(0))}(\Psi(t)-\Psi(0))^{-\alpha} E_{1,\, 1-\alpha}\Big(\lambda(\Psi(t)-\Psi(0)\Big), \\
\prescript{TC}{0}{\mathscr{D}}_{\Psi(t)}^{\alpha, \lambda}(1)=\lambda^n e^{-\lambda(\Psi(t)-\Psi(0)} (\Psi(t)-\Psi(0)^{n-\alpha} E_{1,\, n+1-\alpha}\Big(\lambda(\Psi(t)-\Psi(0)\Big),
\end{align*}
where $n-1<\alpha<n$ and $E_{a,b}$ is the two-parameter Mittag-Leffler function \cite[Chapter 4]{Gorenflo}.
\end{proposition}

\begin{proof}
Setting $\xi=0$ in the result of Proposition \ref{h1}, for the first two results (Riemann--Liouville type) it suffices to note that
\[
\prescript{}{1}F_1(1;c;z)=\sum_{n=0}^{\infty}\frac{\Gamma(1+n)\Gamma(c)}{\Gamma(1)\Gamma(c+n)}\cdot\frac{z^n}{n!}=\sum_{n=0}^{\infty}\frac{z^n\Gamma(c)}{\Gamma(c+n)}=\Gamma(c)E_{1,c}(z),
\]
for any $z\in\mathbb{C}$ and $c\not\in\mathbb{Z}^-_0$ (the restriction on $c$ can be removed upon dividing by $\Gamma(c)$).

The proof of the third result (Caputo type) is a little more complicated, and requires the fact that the $n^{th}$ derivative of a power function of order $k<n$ is zero when $k,n\in\N$:
\begin{align*}
\prescript{TC}{0}{\mathscr{D}}_{\Psi(t)}^{\alpha, \lambda}\,(1)&= e^{-\lambda\left(\Psi(t)-\Psi(0) \right)}\sum_{k=0}^{\infty}\, \frac{\lambda^k}{k!}\,\prescript{C}{0}{\mathscr{D}}_{\Psi(t)}^{\alpha}\left(\Psi(t)-\Psi(0) \right)^{k} \\
&=e^{-\lambda\left(\Psi(t)-\Psi(0) \right)}\sum_{k=n}^{\infty}\, \frac{\lambda^k}{\Gamma(k-\alpha+1)}\,\left(\Psi(t)-\Psi(0) \right)^{k-\alpha}\\
&=e^{-\lambda\left(\Psi(t)-\Psi(0) \right)}\sum_{k=0}^{\infty}\, \frac{\lambda^{k+n}}{\Gamma(k+n-\alpha+1)}\,\left(\Psi(t)-\Psi(0) \right)^{k+n-\alpha},
\end{align*}
where the first $n$ terms of the series disappeared due to the above-mentioned fact.
\end{proof}  

\begin{rem}
It is interesting to observe that the $\Psi$-tempered derivative of Caputo type does not give zero when applied to a constant function, although one of the key properties of the classical Caputo derivative is that it sends constants to zero. The reason is that the function $y(t)=1$ is not playing the role of a constant (zeroth power function) in $\Psi$-tempered fractional calculus: in this setting, the function corresponding to a constant would be $e^{-\lambda\Psi(t)}$, and this function is indeed mapped to zero by any Caputo-type $\Psi$-tempered derivative, as proved in Theorem \ref{Thm:kernelC} below.
\end{rem}

\subsection{Composition relations}

We have already seen above some composition properties of the $\Psi$-tempered fractional integrals and derivatives, namely every case where there is an exact semigroup property, the composition of two such operators being precisely another one. From \cite[Theorem 3.15]{Fahad} we also have the following inversion property, valid for $\alpha\in(0,\infty)$ and $n-1\leq\alpha<n\in\mathbb{N}$ and $\lambda\in\mathbb{R}$ and any function $y$ such that this expression exists:
\begin{multline}
\label{ID}
\prescript{T}{0}{\mathscr{I}}_{\Psi(t)}^{\alpha, \lambda}\,\prescript{TR}{0}{\mathscr{D}}_{\Psi(t)}^{\alpha, \lambda}\,y(t) \\
=y(t)-e^{-\lambda\Psi(t)}\sum_{k=1}^{n}\frac{(\Psi(t)-\Psi(0))^{\alpha-k}}{\Gamma(\alpha-k+1)}\left[\left(\frac{1}{{\Psi}^{'}(t)}\cdot\frac{\mathrm{d}}{\mathrm{d}t}\right)^{n-k}\prescript{}{0}{\mathscr{I}}_{\Psi(t)}^{n-\alpha} \left( e^{\lambda\Psi(t)}y(t)\right) \right]_{t=0}.
\end{multline}
As a special case, putting $\alpha=m\in\mathbb{N}$, we obtain the following inversion property for non-fractional $\Psi$-tempered integrals and derivatives:
\begin{equation}\label{861}
\prescript{T}{0}{\mathscr{I}}_{\Psi(t)}^{m, \lambda} \left[ \prescript{T}{}{\mathscr{D}}_{\Psi(t)}^{m, \lambda}\,y(t)\right]= y(t)- e^{-\lambda(\Psi(t)-\Psi(0))} \sum_{k=0}^{m-1}\frac{(\Psi(t)-\Psi(0))^k}{k!}\left[\prescript{T}{}{\mathscr{D}}_{\Psi(t)}^{k, \lambda}\,y(t) \right]_{t=0},
\end{equation}
obtained by noticing that $n=\alpha+1=m+1$ in this case but the $(m+1-k)$th derivative of the $1$st integral in the square brackets is simply the $(m-k)$th derivative, and then multiplying and dividing by the constant $e^{\lambda\Psi(0)}$ and replacing $k$ by $m-k$ in the sum to obtain \eqref{861}.

We continue with some more results on compositions of $\Psi$-tempered operators in various cases where a direct semigroup property is not valid.

\begin{proposition}\label{58}
If $\alpha\in(0, \infty)$, $\lambda\in\mathbb{R}$ and $m\in\N$, then for any $y\in AC_{\Psi}^{m+n}[0, b]$ where $n-1<\alpha<n\in\N$, we have:
\begin{multline*}
\prescript{TR}{0}{\mathscr{D}}_{\Psi(t)}^{\alpha, \lambda}\,\left[ \prescript{T}{}{\mathscr{D}}_{\Psi(t)}^{m, \lambda} y(t)\right] =\prescript{TR}{0}{\mathscr{D}}_{\Psi(t)}^{\alpha+m, \lambda}y(t) \\
- e^{-\lambda(\Psi(t)-\Psi(0))} \sum_{k=0}^{m-1}\frac{(\Psi(t)-\Psi(0))^{k-\alpha-m}}{\Gamma(k+1-\alpha-m)}\left[\prescript{T}{}{\mathscr{D}}_{\Psi(t)}^{k, \lambda} y(t) \right]_{t=0}.
\end{multline*}
\end{proposition}

\begin{proof}
Using \eqref{DI} and then \eqref{861} and then Proposition \eqref{Prop:specfunc1}, we have
\begin{align*}
\prescript{TR}{0}{\mathscr{D}}_{\Psi(t)}^{\alpha, \lambda}\,\left[ \prescript{T}{}{\mathscr{D}}_{\Psi(t)}^{m, \lambda} y(t)\right] &= \prescript{TR}{0}{\mathscr{D}}_{\Psi(t)}^{\alpha+m, \lambda}\,\prescript{T}{0}{\mathscr{I}}_{\Psi(t)}^{m, \lambda}\left[ \prescript{T}{}{\mathscr{D}}_{\Psi(t)}^{m, \lambda} y(t)\right] \\
&\hspace{-2cm}=\prescript{TR}{0}{\mathscr{D}}_{\Psi(t)}^{\alpha+m, \lambda}\left[y(t)- e^{-\lambda(\Psi(t)-\Psi(0))} \sum_{k=0}^{m-1}\frac{(\Psi(t)-\Psi(0))^k}{k!}\left[\prescript{T}{}{\mathscr{D}}_{\Psi(t)}^{k, \lambda}\,y(t) \right]_{t=0}\right] \\
&\hspace{-2cm}=\prescript{TR}{0}{\mathscr{D}}_{\Psi(t)}^{\alpha+m, \lambda}y(t) \\ &- \sum_{k=0}^{m-1}\left[\prescript{T}{}{\mathscr{D}}_{\Psi(t)}^{k, \lambda}\,y(t) \right]_{t=0}\cdot\prescript{TR}{0}{\mathscr{D}}_{\Psi(t)}^{\alpha+m, \lambda}\left[e^{-\lambda(\Psi(t)-\Psi(0))}\frac{(\Psi(t)-\Psi(0))^k}{k!}\right]\\
&\hspace{-2cm}=\prescript{TR}{0}{\mathscr{D}}_{\Psi(t)}^{\alpha+m, \lambda}y(t)- \sum_{k=0}^{m-1}\left[\prescript{T}{}{\mathscr{D}}_{\Psi(t)}^{k, \lambda}\,y(t) \right]_{t=0}\cdot e^{-\lambda(\Psi(t)-\Psi(0))}\frac{(\Psi(t)-\Psi(0))^{k-\alpha-m}}{\Gamma(k-\alpha-m+1)},
\end{align*}
which is the claimed result.
\end{proof}

\begin{theorem} \label{Thm:RLCrel}
If $\alpha\in(0,\infty)$, $\lambda\in\mathbb{R}$ and $n-1<\alpha<n\in\N$, then any $y\in AC_{\Psi}^n[0, b]$ satisfies the following relationship between its $\Psi$-tempered derivatives of Riemann--Liouville and Caputo types:
\begin{multline*}
\prescript{TC}{0}{\mathscr{D}}_{\Psi(t)}^{\alpha, \lambda} y(t)=\prescript{TR}{0}{\mathscr{D}}_{\Psi(t)}^{\alpha, \lambda} \left\lbrace  y(t) - \sum_{k=0}^{n-1} \frac{e^{-\lambda(\Psi(t)-\Psi(0))}(\Psi(t)-\Psi(0))^k}{k!}\left[\prescript{T}{}{\mathscr{D}}_{\Psi(t)}^{k,\lambda}y(t)\right]_{t=0}\right\rbrace \\
=\prescript{TR}{0}{\mathscr{D}}_{\Psi(t)}^{\alpha, \lambda} \left\lbrace  y(t) - \sum_{k=0}^{n-1} \frac{e^{-\lambda\Psi(t)}(\Psi(t)-\Psi(0))^k}{k!}\left[\left(\frac{1}{\Psi'(t)}\cdot\frac{\mathrm{d}}{\mathrm{d}t} \right)^k\left(e^{\lambda\Psi(t)} y(t) \right)\right]_{t=0}\right\rbrace.
\end{multline*}
\end{theorem}

\begin{proof}
Applying \cite[Theorem 3]{Almeida} (which is itself an immediate consequence of \cite[Theorem 3.1]{Diethelm} together with the $\Psi$-convolution relation) with $f(t)=e^{\lambda(\Psi(t)-\Psi(0))} y(t)$, we obtain
\begin{multline}\label{815a}
\prescript{C}{0}{\mathscr{D}}_{\Psi(t)}^{\alpha} \left(e^{\lambda(\Psi(t)-\Psi(0))} y(t)\right)=\prescript{RL}{0}{\mathscr{D}}_{\Psi(t)}^{\alpha} \Bigg\lbrace e^{\lambda(\Psi(t)-\Psi(0))} y(t) \\ -\sum_{k=0}^{n-1} \frac{(\Psi(t)-\Psi(0))^k}{k!}\left[ \left( \frac{1}{\Psi'(t)}\cdot\frac{\mathrm{d}}{\mathrm{d}t} \right)^k\left(e^{\lambda(\Psi(t)-\Psi(0))} y(t) \right)\right]_{t=0}\Bigg\rbrace.
\end{multline}
Therefore, using the expression \eqref{Tprod:C} for the $\Psi$-tempered derivative,
\begin{align*}
\prescript{TC}{0}{\mathscr{D}}_{\Psi(t)}^{\alpha, \lambda} y(t)&=e^{-\lambda(\Psi(t)-\Psi(0))}\prescript{C}{0}{\mathscr{D}}_{\Psi(t)}^{\alpha} \left(e^{\lambda(\Psi(t)-\Psi(0))} y(t)\right) \\
&=e^{-\lambda(\Psi(t)-\Psi(0))}\prescript{RL}{0}{\mathscr{D}}_{\Psi(t)}^{\alpha} \Bigg\lbrace e^{\lambda(\Psi(t)-\Psi(0))} y(t) \\ &\hspace{3cm}-\sum_{k=0}^{n-1} \frac{(\Psi(t)-\Psi(0))^k}{k!}\left[\prescript{T}{}{\mathscr{D}}_{\Psi(t)}^{k,\lambda}y(t)\right]_{t=0}\Bigg\rbrace \\
&=e^{-\lambda(\Psi(t)-\Psi(0))}\prescript{RL}{0}{\mathscr{D}}_{\Psi(t)}^{\alpha}\Bigg\lbrace e^{\lambda(\Psi(t)-\Psi(0))}\Bigg[ y(t) \\ &\hspace{2cm}-\sum_{k=0}^{n-1}e^{-\lambda(\Psi(t)-\Psi(0))}\frac{(\Psi(t)-\Psi(0))^k}{k!}\left[\prescript{T}{}{\mathscr{D}}_{\Psi(t)}^{k,\lambda}y(t)\right]_{t=0}\Bigg]\Bigg\rbrace,
\end{align*}
and the result follows immediately by \eqref{Tprod:R}.
\end{proof}

\begin{rem}
As special cases of Theorem \ref{Thm:RLCrel}, we can recover already known results: on fractional calculus with respect to functions \cite[Theorem 3]{Almeida}, by setting $\lambda=0$; on tempered fractional calculus \cite[Proposition 1]{Li}, by setting $\Psi(t)=t$; on the original Riemann--Liouville and Caputo fractional calculus \cite[Theorem 3.1]{Diethelm}, by setting $\lambda=0$ and $\Psi(t)=t$.
\end{rem}

\begin{theorem}
If $\alpha, \beta\in(0, \infty)$ and $\lambda\in\mathbb{R}$ such that $\lfloor\beta\rfloor<\lfloor\alpha\rfloor$ (this is a stronger version of $\beta<\alpha$), then
\[
\prescript{TC}{0}{\mathscr{D}}_{\Psi(t)}^{\beta,\lambda}\prescript{T}{0}{\mathscr{I}}_{\Psi(t)}^{\alpha, \lambda}y(t)=\prescript{T}{0}{\mathscr{I}}_{\Psi(t)}^{\alpha-\beta, \lambda}y(t),\qquad t\in(0, b].
\]
\end{theorem}

\begin{proof}
Say $m-1\leq\beta<m$ and $n-1\leq\alpha<n$, so that our assumption is $m<n$ and $m\leq\alpha$. Using the semigroup properties \eqref{DI} and \eqref{semi:I}, we have
\begin{align*}
\prescript{TC}{0}{\mathscr{D}}_{\Psi(t)}^{\beta,\lambda}\prescript{T}{0}{\mathscr{I}}_{\Psi(t)}^{\alpha, \lambda}y(t)&=\prescript{T}{0}{\mathscr{I}}_{\Psi(t)}^{m-\beta, \lambda}\,\prescript{T}{}{\mathscr{D}}_{\Psi(t)}^{m, \lambda}\Big[\prescript{T}{0}{\mathscr{I}}_{\Psi(t)}^{\alpha, \lambda}y(t)\Big] \\
&=\prescript{T}{0}{\mathscr{I}}_{\Psi(t)}^{m-\beta, \lambda}\Big[\prescript{T}{0}{\mathscr{I}}_{\Psi(t)}^{\alpha-m, \lambda}y(t)\Big]=\prescript{T}{0}{\mathscr{I}}_{\Psi(t)}^{(m-\beta)+(\alpha-m), \lambda}y(t),
\end{align*}
giving the result.
\end{proof}

\begin{theorem}\label{c1}
Let $\alpha\in(0,\infty)$ and $\lambda\in\mathbb{R}$.
\begin{enumerate}[(a)]
\item If $n-1<\alpha<n\in\N$ and $y\in C^n[0, b]$, then $\prescript{TC}{0}{\mathscr{D}}_{\Psi(t)}^{\alpha, \lambda}\,\prescript{T}{0}{\mathscr{I}}_{\Psi(t)}^{\alpha, \lambda}y(t)=y(t)$.
\item If $n-1<\alpha<n\in\N$ and $y\in AC_{\Psi}^n[0, b]$, then
\begin{multline*}
\prescript{T}{0}{\mathscr{I}}_{\Psi(t)}^{\alpha, \lambda}\,\prescript{TC}{0}{\mathscr{D}}_{\Psi(t)}^{\alpha, \lambda}\,y(t)=y(t)- e^{-\lambda(\Psi(t)-\Psi(0))}\sum_{k=0}^{n-1}\frac{(\Psi(t)-\Psi(0))^{k}}{k!}\left[\prescript{T}{}{\mathscr{D}}_{\Psi(t)}^{k,\lambda}\,y(t)\right]_{t=0} \\
=y(t)- e^{-\lambda\Psi(t)}\sum_{k=0}^{n-1}\frac{(\Psi(t)-\Psi(0))^{k}}{k!}\left[\left(\frac{1}{\Psi'(t)}\cdot\frac{\mathrm{d}}{\mathrm{d}t} \right)^k\left(e^{\lambda\Psi(t)} y(t) \right)\right]_{t=0}.
\end{multline*}
\end{enumerate}
\end{theorem}

\begin{proof}
We prove each of the two statements separately, noting that the first one requires a stronger condition on the function $y$.
\begin{enumerate}[(a)]
\item Using the definition of Caputo-type derivatives and then the relations \eqref{DI} and \eqref{ID} respectively, we have:
\begin{align*}
\prescript{TC}{0}{\mathscr{D}}_{\Psi(t)}^{\alpha, \lambda}\,\prescript{T}{0}{\mathscr{I}}_{\Psi(t)}^{\alpha, \lambda}\,y(t)&=\prescript{T}{0}{\mathscr{I}}_{\Psi(t)}^{n-\alpha, \lambda}\,\prescript{T}{}{\mathscr{D}}_{\Psi(t)}^{n, \lambda}\,\prescript{T}{0}{\mathscr{I}}_{\Psi(t)}^{\alpha, \lambda}\,y(t) \\
&=\prescript{T}{0}{\mathscr{I}}_{\Psi(t)}^{n-\alpha, \lambda}\,\,\prescript{TR}{0}{\mathscr{D}}_{\Psi(t)}^{n-\alpha, \lambda}\,y(t) \\
&=y(t)-e^{-\lambda\Psi(t)}\frac{(\Psi(t)-\Psi(0))^{n-\alpha-1}}{\Gamma(n-\alpha)}\left[\prescript{}{0}{\mathscr{I}}_{\Psi(t)}^{1-\alpha+n} \left( e^{\lambda\Psi(t)}y(t)\right) \right]_{t=0},
\end{align*}
where this holds for any $y$ such that the expression is defined. If we then assume $y$ is continuous, then the initial-value ($t=0$) part of this expression is zero by \eqref{limit0:int}, and the stated result follows.

\item Using the definition of Caputo-type derivatives and then the semigroup property \eqref{semi:I}, we have:
\begin{align*}
\prescript{T}{0}{\mathscr{I}}_{\Psi(t)}^{\alpha, \lambda}\,\prescript{TC}{0}{\mathscr{D}}_{\Psi(t)}^{\alpha, \lambda}\,y(t)&=\prescript{T}{0}{\mathscr{I}}_{\Psi(t)}^{\alpha, \lambda}\,\prescript{T}{0}{\mathscr{I}}_{\Psi(t)}^{n-\alpha, \lambda}\,\prescript{T}{}{\mathscr{D}}_{\Psi(t)}^{n, \lambda}\,y(t) \\
&=\prescript{T}{0}{\mathscr{I}}_{\Psi(t)}^{n, \lambda}\,\prescript{T}{}{\mathscr{D}}_{\Psi(t)}^{n, \lambda}\,y(t),
\end{align*}
and then the result follows from \cite[Theorem 4]{Almeida}.
\end{enumerate}
\end{proof}

\begin{proposition}
If $\alpha\in(0,\infty)$ with $n-1<\alpha<n\in\N$, $\lambda\in\mathbb{R}$ and $m\in\N$, then
\[
\prescript{TC}{0}{\mathscr{D}}_{\Psi(t)}^{\alpha, \lambda}\,\prescript{T}{}{\mathscr{D}}_{\Psi(t)}^{m, \lambda}\,y(t)=\prescript{TC}{0}{\mathscr{D}}_{\Psi(t)}^{\alpha+m, \lambda}\,y(t),
\]
for any function $y$ such that these expressions are well-defined, such as $y\in AC_{\Psi}^{m+n}[0, b]$.
\end{proposition}

\begin{proof}
This follows from the definition of Caputo-type derivatives, where $n-1<\alpha<n$ implies $m+n-1<m+\alpha<m+n$, and the semigroup property of standard non-fractional derivatives that is immediate from their definition:
\begin{align*}
\prescript{TC}{0}{\mathscr{D}}_{\Psi(t)}^{\alpha, \lambda}\,\prescript{T}{}{\mathscr{D}}_{\Psi(t)}^{m, \lambda}\,y(t)&=\prescript{T}{0}{\mathscr{I}}_{\Psi(t)}^{n-\alpha, \lambda}\,\prescript{T}{}{\mathscr{D}}_{\Psi(t)}^{n, \lambda}\,\prescript{T}{}{\mathscr{D}}_{\Psi(t)}^{m, \lambda}\,y(t) \\
&=\prescript{T}{0}{\mathscr{I}}_{\Psi(t)}^{(m+n)-(m+\alpha), \lambda}\,\prescript{T}{}{\mathscr{D}}_{\Psi(t)}^{m+n, \lambda}y(t)=\prescript{TC}{0}{\mathscr{D}}_{\Psi(t)}^{\alpha+m, \lambda} y(t).
\end{align*}
\end{proof}

\begin{proposition}
If $\alpha\in(0,\infty)$ with $n-1<\alpha<n\in\N$ and $\lambda\in\mathbb{R}$ and $m\in\N$, then for any $y\in AC_{\Psi}^{m+n}[0, b]$,
\begin{align*}
\prescript{T}{}{\mathscr{D}}_{\Psi(t)}^{m, \lambda}\,\prescript{TC}{0}{\mathscr{D}}_{\Psi(t)}^{\alpha, \lambda}\, y(t)&=\prescript{TC}{0}{\mathscr{D}}_{\Psi(t)}^{\alpha+m, \lambda}\,y(t) \\ &\hspace{1cm}+ e^{-\lambda \left( \Psi(t)-\Psi(0)\right) }\sum_{k=0}^{m-1} \frac{\left( \Psi(t)-\Psi(0)\right)^{k+n-\alpha-m}}{\Gamma\left( k+n-\alpha-m+1\right) }\left[ \prescript{T}{}{\mathscr{D}}_{\Psi(t)}^{k+n,\lambda}y(t)\right]_{t=0} \\
&\hspace{-2cm}=\prescript{TC}{0}{\mathscr{D}}_{\Psi(t)}^{\alpha+m, \lambda}\,y(t) \\ &\hspace{-1cm}+ e^{-\lambda\Psi(t)}\sum_{k=0}^{m-1} \frac{\left( \Psi(t)-\Psi(0)\right)^{k+n-\alpha-m}}{\Gamma\left( k+n-\alpha-m+1\right) }\left[ \left( \frac{1}{\Psi'(t)}\cdot\frac{\mathrm{d}}{\mathrm{d}t}\right)^{k+n}\left(e^{\lambda\Psi(t)} y(t) \right)\right]_{t=0}.
\end{align*}
\end{proposition}

\begin{proof}
This follows immediately from \cite[Theorem 9]{Almeida} upon multiplying and dividing by the appropriate exponential factors.
\end{proof}

\subsection{Kernels of the derivative operators}

In this subsection, we discover the conditions under which a function's $\Psi$-tempered fractional derivative (of either Riemann--Liouville or Caputo type) is zero. Equivalently (since the operators are linear), we discover the conditions under which two functions have equal $\Psi$-tempered fractional derivatives.

\begin{theorem} \label{Thm:kernelR}
Let $\alpha\in(0,\infty)$ and $\lambda\in\mathbb{R}$ with $n-1<\alpha<n\in\N$. Then, two functions $f,g\in AC_{\Psi}^n[0, b]$ satisfy
\begin{equation}
\label{kernel:TR}
\prescript{TR}{0}{\mathscr{D}}_{\Psi(t)}^{\alpha, \lambda} f(t)=\prescript{TR}{0}{\mathscr{D}}_{\Psi(t)}^{\alpha, \lambda} g(t),\qquad t\in(0, b],
\end{equation}
if and only if their difference can be written in the form
\begin{equation}\label{571}
f(t)-g(t)=e^{-\lambda\Psi(t)} \sum_{k=1}^{n} c_k (\Psi(t)-\Psi(0))^{\alpha-k},
\end{equation}
where $c_1,c_2,\cdots,c_n$ are constants given by
\begin{equation}\label{572}
c_k=\frac{1}{\Gamma(\alpha-k+1)}\left[\left( \frac{1}{\Psi'(t)}\cdot\frac{\mathrm{d}}{\mathrm{d}t}\right)^{n-k}\prescript{T}{0}{\mathscr{I}}_{\Psi(t)}^{n-\alpha}\Big(e^{\lambda\Psi(t)}\left[f(t)-g(t)\right]\Big) \right]_{t=0}.
\end{equation}
\end{theorem}

\begin{proof}
If \eqref{kernel:TR} holds, then $\prescript{TR}{0}{\mathscr{D}}_{\Psi(t)}^{\alpha, \lambda} \left[ f(t)-g(t)\right] =0$; applying the operator $\prescript{T}{0}{\mathscr{I}}_{\Psi(t)}^{\alpha, \lambda}$ on both sides, the result follows immediately from \eqref{ID}.

Conversely, if \eqref{571} holds, then applying the operator $\prescript{TR}{0}{\mathscr{D}}_{\Psi(t)}^{\alpha, \lambda}$ on both sides of \eqref{571} gives
\[
\prescript{TR}{0}{\mathscr{D}}_{\Psi(t)}^{\alpha, \lambda} \left[ f(t)-g(t)\right] = \sum_{k=1}^{n} c_k\,\prescript{TR}{0}{\mathscr{D}}_{\Psi(t)}^{\alpha, \lambda} \left[ e^{-\lambda\Psi(t)}(\Psi(t)-\Psi(0))^{\alpha-k}\right].
\]
By Proposition \ref{Prop:specfunc1} with $\xi=\alpha-k\geq\alpha-n>-1$ (here we use the assumption that $n-1<\alpha<n$) we have
\begin{align*}
\prescript{TR}{0}{\mathscr{D}}_{\Psi(t)}^{\alpha, \lambda} \left[ f(t)-g(t)\right] &= \sum_{k=1}^{n} c_k\,\frac{\Gamma(\alpha-k+1)}{\Gamma(-k+1)}\,e^{-\lambda\Psi(t)}(\Psi(t)-\Psi(0))^{-k+1} \\
&=\sum_{k=1}^nc_k(0)e^{-\lambda\Psi(t)}(\Psi(t)-\Psi(0))^{-k+1}=0,
\end{align*}
using the fact that the entire function $\frac{1}{\Gamma(z)}$ is zero when $z$ is a non-positive integer. Thus, \eqref{kernel:TR} is established.
\end{proof}

\begin{theorem} \label{Thm:kernelC}
Let $\alpha\in(0,\infty)$ and $\lambda\in\mathbb{R}$ with $n-1<\alpha<n\in\N$. Then, two functions $f,g\in AC_{\Psi}^n[0, b]$ satisfy
\begin{equation}
\label{kernel:TC}
\prescript{TC}{0}{\mathscr{D}}_{\Psi(t)}^{\alpha, \lambda} f(t)=\prescript{TC}{0}{\mathscr{D}}_{\Psi(t)}^{\alpha, \lambda} g(t),\quad t\in(0, b],
\end{equation}
if and only if their difference can be written in the form
\begin{equation}\label{881}
f(t)=g(t)+e^{-\lambda\Psi(t)} \sum_{k=0}^{n-1} c_k (\Psi(t)-\Psi(0))^k,
\end{equation}
where $c_0,c_1,\cdots,c_{n-1}$ are constants given by
\begin{equation}\label{881a}
c_k=\frac{1}{k!}\left[\left( \frac{1}{\Psi'(t)}\cdot\frac{\mathrm{d}}{\mathrm{d}t}\right) ^k \Big( e^{\lambda\Psi(t)}\left[ f(t)-g(t)\right] \Big) \right]_{t=0}.
\end{equation}
\end{theorem}

\begin{proof}
If \eqref{kernel:TC} holds, then $\prescript{TC}{0}{\mathscr{D}}_{\Psi(t)}^{\alpha, \lambda} \left[ f(t)-g(t)\right] =0$; applying the operator $\prescript{T}{0}{\mathscr{I}}_{\Psi(t)}^{\alpha, \lambda}$ on both sides, the result follows immediately from part (b) of Theorem \ref{c1}.

Conversely, if \eqref{881} holds, then applying the operator $\prescript{TC}{0}{\mathscr{D}}_{\Psi(t)}^{\alpha, \lambda}$ on both sides of \eqref{881} gives
\[
\prescript{TC}{0}{\mathscr{D}}_{\Psi(t)}^{\alpha, \lambda} \left[ f(t)-g(t)\right] = \sum_{k=0}^{n-1} c_k\,\prescript{TC}{0}{\mathscr{D}}_{\Psi(t)}^{\alpha, \lambda} \left[ e^{-\lambda\Psi(t)}(\Psi(t)-\Psi(0))^k\right].
\]
Since $k<n$ for all $k$, the $\alpha^{th}$ $\Psi$-tempered derivative of each term in square brackets is zero, so by the definition of the Caputo-type fractional derivative, the whole sum is zero.
\end{proof}

\begin{rem}
Taking $\lambda=0$ in Theorem \ref{Thm:kernelR} and Theorem \ref{Thm:kernelC} gives \cite[Theorem 6]{Almeida}.
\end{rem}

\section{Mean value theorem and Taylor's theorem} \label{Sec:MVTTT}

\subsection{Mean value theorems for $\Psi$-tempered operators}

In this subsection, we present three results concerning the $\Psi$-tempered operators. Each of these can be interpreted as a sort of mean value theorem: firstly for the $\Psi$-tempered fractional integral, and then respectively for the $\Psi$-tempered fractional derivatives of Caputo and Riemann--Liouville type.

\begin{theorem}\label{t22}
Let $\alpha\in(0,\infty)$ and $\lambda\in\mathbb{R}$. If $y\in C[0, b]$ and $t\in(0,b]$, then there exists $c\in(0,t)$ such that
\begin{equation}
\label{MVT:int}
\prescript{T}{0}{\mathscr{I}}_{\Psi(t)}^{\alpha, \lambda}\,y(t) \\ =\left(\Psi(t)-\Psi(0)\right)^{\alpha}e^{-\lambda(\Psi(t)-\Psi(0))} E_{1,\alpha+1}\Big( \lambda(\Psi(t)-\Psi(0))\Big)\cdot y(c).
\end{equation}
\end{theorem}

\begin{proof}
By the definition of $\Psi$-tempered fractional integrals, we have
\[
\prescript{T}{0}{\mathscr{I}}_{\Psi(t)}^{\alpha, \lambda}\,y(t)=\frac{e^{-\lambda(\Psi(t)-\Psi(0))}}{\Gamma(\alpha)}\int_{0}^{t} \Psi'(s) \left(\Psi(t)-\Psi(s)\right)^{\alpha-1}e^{\lambda(\Psi(s)-\Psi(0))}\,y(s)\,\mathrm{d}s.
\]
Since $y$ is continuous, we can use the first mean value theorem for integrals to obtain that there exists $c\in(0,b)$ such that
\[
\prescript{T}{0}{\mathscr{I}}_{\Psi(t)}^{\alpha, \lambda}\,y(t)=y(c)\cdot\prescript{T}{0}{\mathscr{I}}_{\Psi(t)}^{\alpha,\lambda}\,(1),
\]
and then Proposition \ref{Coroll:1} gives the result.
\end{proof}

\begin{rem}
By taking $\lambda=0$ in Theorem \ref{t22}, we obtain \cite[Theorem 7]{Almeida}.
\end{rem}

\begin{cor}
If $\alpha\in(0,1)$, $\lambda\in\mathbb{R}$ and $y\in C^1[0, b]$, then for all $t\in(0, b]$, there exists some $c\in (0, t)$ such that
\[
\prescript{TC}{0}{\mathscr{D}}_{\Psi(t)}^{\alpha, \lambda}\,y(c)=\frac{e^{\lambda(\Psi(t)-\Psi(0))}y(t)-y(0)}{(\Psi(t)-\Psi(0))^\alpha E_{1,\, \alpha+1}(\lambda(\Psi(t)-\Psi(0)))}.
\]
\end{cor}

\begin{proof}
This follows from substituting $\prescript{TC}{0}{\mathscr{D}}_{\Psi(t)}^{\alpha, \lambda}\,y$ instead of $y$ in Theorem \ref{t22} and using part (b) of Theorem \ref{c1} to simplify the left-hand side of \eqref{MVT:int}.
\end{proof}

\begin{cor}
If $\alpha\in(0,1)$ and $\lambda\in\mathbb{R}$ and $y\in C^1[0, b]$, then for all $t\in(0, b]$, there exists some $c\in (0, t)$ such that
\[
\prescript{TR}{0}{\mathscr{D}}_{\Psi(t)}^{\alpha, \lambda}\,y(c)=\frac{e^{\lambda(\Psi(t)-\Psi(0))}y(t)-\frac{(\Psi(t)-\Psi(0))^{\alpha-1}}{\Gamma(\alpha)}\left[ \prescript{}{0}{\mathscr{I}}_{\Psi(t)}^{1-\alpha}\left( e^{\lambda(\Psi(t)-\Psi(0))}y(t)\right) \right]_{t=0} }{(\Psi(t)-\Psi(0))^\alpha E_{1,\, \alpha+1}(\lambda(\Psi(t)-\Psi(0)))}.
\]
\end{cor}

\begin{proof}
This follows from substituting $\prescript{TR}{0}{\mathscr{D}}_{\Psi(t)}^{\alpha, \lambda}\,y$ instead of $y$ in Theorem \ref{t22} and using \eqref{ID} to simplify the left-hand side of \eqref{MVT:int}.
\end{proof}

\subsection{Taylor's theorem in the setting of $\Psi$-tempered fractional derivatives}

In this subsection, we present two results which can be interpreted as versions of Taylor's theorem for the $\Psi$-tempered fractional derivative operators, of Caputo and Riemann--Liouville type respectively. The proofs are following the lines of \cite{trujillo-rivero-bonilla,odibat-shawagfeh}.

\begin{theorem}
Let $\alpha\in(0, 1)$, $\lambda\in\mathbb{R}$, $n\in\N$, and  $f$ be any function such that $\prescript{TC}{0}{\mathscr{D}}_{\Psi(t)}^{k\alpha, \lambda}f$ exists and is continuous for all $k=0, 1,\, 2, \cdots, n+1$. Then, for all $t\in[0, b]$, we have
\begin{multline*}
f(t)=e^{-\lambda(\Psi(t)-\Psi(0))}\Bigg\lbrace \sum_{k=0}^{n} \frac{(\Psi(t)-\Psi(0))^{k\alpha}}{\Gamma(k\alpha+1)}\cdot\prescript{TC}{0}{\mathscr{D}}_{\Psi(t)}^{k\alpha, \lambda}f(0)\\
+(\Psi(t)-\Psi(0))^{(n+1)\alpha}E_{1,\, (n+1)\alpha+1}\Big(\lambda(\Psi(t)-\Psi(0))\Big)\cdot\left(\prescript{TC}{0}{\mathscr{D}}_{\Psi(t)}^{\alpha, \lambda}\right)^{n+1}f(c)\Bigg\rbrace,
\end{multline*}
for some $c\in(0,t)$.
\end{theorem}

\begin{proof}
Firstly, we find the difference between the functions given by \[\left(\prescript{T}{0}{\mathscr{I}}_{\Psi(t)}^{\alpha, \lambda}\right)^k\left(\prescript{TC}{0}{\mathscr{D}}_{\Psi(t)}^{\alpha, \lambda}\right)^kf(t)\] for different values of $k$. Using the semigroup property \eqref{semi:I} for $\Psi$-tempered integrals and part (b) of Theorem \ref{c1} in the case $n=1$, we obtain:
\begin{align*}
&\left(\prescript{T}{0}{\mathscr{I}}_{\Psi(t)}^{\alpha, \lambda}\right)^k\left(\prescript{TC}{0}{\mathscr{D}}_{\Psi(t)}^{\alpha, \lambda}\right)^kf(t)-\left(\prescript{T}{0}{\mathscr{I}}_{\Psi(t)}^{\alpha, \lambda}\right)^{k+1}\left(\prescript{TC}{0}{\mathscr{D}}_{\Psi(t)}^{\alpha, \lambda}\right)^{k+1}f(t) \\
&\hspace{1cm}=\prescript{T}{0}{\mathscr{I}}_{\Psi(t)}^{k\alpha, \lambda}\left[ \left(\prescript{TC}{0}{\mathscr{D}}_{\Psi(t)}^{\alpha, \lambda}\right)^kf(t) -\prescript{T}{0}{\mathscr{I}}_{\Psi(t)}^{\alpha, \lambda}\,\prescript{TC}{0}{\mathscr{D}}_{\Psi(t)}^{\alpha, \lambda}\left(\prescript{TC}{0}{\mathscr{D}}_{\Psi(t)}^{\alpha, \lambda}\right)^kf(t)\right] \\
&\hspace{1cm}=\prescript{T}{0}{\mathscr{I}}_{\Psi(t)}^{k\alpha, \lambda}\Bigg[ \left(\prescript{TC}{0}{\mathscr{D}}_{\Psi(t)}^{\alpha, \lambda}\right)^kf(t) \\ &\hspace{3cm}-\left\lbrace\left(\prescript{TC}{0}{\mathscr{D}}_{\Psi(t)}^{\alpha, \lambda}\right)^kf(t)- e^{-\lambda(\Psi(t)-\Psi(0))}\cdot\left(\prescript{TC}{0}{\mathscr{D}}_{\Psi(t)}^{\alpha, \lambda}\right)^kf(0)\right\rbrace  \Bigg] \\
&\hspace{1cm}=\bigg(\prescript{T}{0}{\mathscr{I}}_{\Psi(t)}^{k\alpha, \lambda} e^{-\lambda(\Psi(t)-\Psi(0))}\bigg)\cdot\left(\prescript{TC}{0}{\mathscr{D}}_{\Psi(t)}^{\alpha, \lambda}\right)^kf(0) \\
&\hspace{1cm}=\bigg(e^{-\lambda(\Psi(t)-\Psi(0))}\cdot\prescript{}{0}{\mathscr{I}}_{\Psi(t)}^{k\alpha}\,(1) \bigg)\cdot\left(\prescript{TC}{0}{\mathscr{D}}_{\Psi(t)}^{\alpha, \lambda}\right)^kf(0) \\
&\hspace{1cm}= e^{-\lambda(\Psi(t)-\Psi(0))}\,\frac{(\Psi(t)-\Psi(0))^{k\alpha}}{\Gamma(k\alpha+1)}\cdot\left(\prescript{TC}{0}{\mathscr{D}}_{\Psi(t)}^{\alpha, \lambda}\right)^kf(0).
\end{align*}
Summing this result over $k$ from $0$ to $n$, we obtain
\begin{multline} \label{834}
f(t)-\left(\prescript{T}{0}{\mathscr{I}}_{\Psi(t)}^{\alpha, \lambda}\right)^{n+1}\left(\prescript{TC}{0}{\mathscr{D}}_{\Psi(t)}^{\alpha, \lambda}\right)^{n+1}f(t) \\ =e^{-\lambda(\Psi(t)-\Psi(0))}\,\sum_{k=0}^{n} \,\frac{(\Psi(t)-\Psi(0))^{k\alpha}}{\Gamma(k\alpha+1)}\cdot\left(\prescript{TC}{0}{\mathscr{D}}_{\Psi(t)}^{\alpha, \lambda}\right)^kf(0).
\end{multline}
By Theorem \ref{t22}, there exists $c\in(0,t)$ such that
\begin{multline} \label{835}
\left(\prescript{T}{0}{\mathscr{I}}_{\Psi(t)}^{\alpha, \lambda}\right)^{n+1}\left(\prescript{TC}{0}{\mathscr{D}}_{\Psi(t)}^{\alpha, \lambda}\right)^{n+1}f(t) \\ =\left(\Psi(t)-\Psi(0)\right)^{(n+1)\alpha}e^{-\lambda(\Psi(t)-\Psi(0))} E_{1,(n+1)\alpha+1}\Big( \lambda(\Psi(t)-\Psi(0))\Big)\cdot \left(\prescript{TC}{0}{\mathscr{D}}_{\Psi(t)}^{\alpha, \lambda}\right)^{n+1}f(c).
\end{multline}
Using the equation \eqref{835} in the equation \eqref{834}, we obtain the stated result.
\end{proof}

\begin{theorem}
Let $\alpha\in(0, 1)$, $\lambda\in\mathbb{R}$, $n\in\N$, and  $f$ be any function such that $\prescript{TR}{0}{\mathscr{D}}_{\Psi(t)}^{k\alpha, \lambda}f$ exists and is continuous for all $k=0, 1,\, 2, \cdots, n+1$. Then, for all $t\in[0, b]$, we have
\begin{multline*}
f(t)=e^{-\lambda(\Psi(t)-\Psi(0))}\left\lbrace \sum_{k=0}^{n} \frac{(\Psi(t)-\Psi(0))^{(k+1)\alpha-1}}{\Gamma((k+1)\alpha)} \left[ \prescript{T}{0}{\mathscr{I}}_{\Psi(t)}^{1-\alpha}\left( e^{\lambda(\Psi(t)-\Psi(0))}\left(\prescript{TR}{0}{\mathscr{D}}_{\Psi(t)}^{\alpha, \lambda}\right)^kf(t)\right) \right]_{t=0}\right.\\
\left.+(\Psi(t)-\Psi(0))^{(n+1)\alpha}E_{1,\, (n+1)\alpha+1}\Big(\lambda(\Psi(t)-\Psi(0))\Big)\cdot\left(\prescript{TR}{0}{\mathscr{D}}_{\Psi(t)}^{\alpha, \lambda}\right)^{n+1}f(c)\right\rbrace,
\end{multline*}
for some $c\in(0,t)$.
\end{theorem}

\begin{proof}
Firstly, we find the difference between the functions given by \[\left(\prescript{T}{0}{\mathscr{I}}_{\Psi(t)}^{\alpha, \lambda}\right)^k\left(\prescript{TR}{0}{\mathscr{D}}_{\Psi(t)}^{\alpha, \lambda}\right)^kf(t)\] for different values of $k$. Using the semigroup property \eqref{semi:I} for $\Psi$-tempered integrals, and then the property \eqref{ID} and Proposition \ref{Prop:specfunc1}, we obtain:
\begin{align*}
&\left(\prescript{T}{0}{\mathscr{I}}_{\Psi(t)}^{\alpha, \lambda}\right)^k\left(\prescript{TR}{0}{\mathscr{D}}_{\Psi(t)}^{\alpha, \lambda}\right)^kf(t)-\left(\prescript{T}{0}{\mathscr{I}}_{\Psi(t)}^{\alpha, \lambda}\right)^{k+1}\left(\prescript{TR}{0}{\mathscr{D}}_{\Psi(t)}^{\alpha, \lambda}\right)^{k+1}f(t) \\
&\hspace{0.5cm}=\prescript{T}{0}{\mathscr{I}}_{\Psi(t)}^{k\alpha, \lambda}\left[ \left(\prescript{TR}{0}{\mathscr{D}}_{\Psi(t)}^{\alpha, \lambda}\right)^kf(t) -\prescript{T}{0}{\mathscr{I}}_{\Psi(t)}^{\alpha, \lambda}\,\prescript{TR}{0}{\mathscr{D}}_{\Psi(t)}^{\alpha, \lambda}\left(\prescript{TR}{0}{\mathscr{D}}_{\Psi(t)}^{\alpha, \lambda}\right)^kf(t)\right] \\
&\hspace{0.5cm}=\prescript{T}{0}{\mathscr{I}}_{\Psi(t)}^{k\alpha, \lambda}\Bigg[ \left(\prescript{TR}{0}{\mathscr{D}}_{\Psi(t)}^{\alpha, \lambda}\right)^kf(t)-\bigg\lbrace\left(\prescript{TR}{0}{\mathscr{D}}_{\Psi(t)}^{\alpha, \lambda}\right)^kf(t) \\ &\hspace{1.5cm}- e^{-\lambda\Psi(t)}\frac{(\Psi(t)-\Psi(0))^{\alpha-1}}{\Gamma(\alpha)}\left[\prescript{}{0}{\mathscr{I}}_{\Psi(t)}^{1-\alpha}\left( e^{\lambda\Psi(t)}\left(\prescript{TR}{0}{\mathscr{D}}_{\Psi(t)}^{\alpha, \lambda}\right)^kf(t)\right) \right]_{t=0}\bigg\rbrace  \Bigg] \\
&\hspace{0.5cm}=\prescript{T}{0}{\mathscr{I}}_{\Psi(t)}^{k\alpha, \lambda}\left(\frac{(\Psi(t)-\Psi(0))^{\alpha-1}}{\Gamma(\alpha)} e^{-\lambda\Psi(t)}\right)\cdot\left[\prescript{}{0}{\mathscr{I}}_{\Psi(t)}^{1-\alpha}\left( e^{\lambda\Psi(t)}\left(\prescript{TR}{0}{\mathscr{D}}_{\Psi(t)}^{\alpha, \lambda}\right)^kf(t)\right) \right]_{t=0} \\
&\hspace{0.5cm}=\frac{(\Psi(t)-\Psi(0))^{(k+1)\alpha-1}}{\Gamma((k+1)\alpha)} e^{-\lambda\Psi(t)}\cdot\left[\prescript{}{0}{\mathscr{I}}_{\Psi(t)}^{1-\alpha}\left( e^{\lambda\Psi(t)}\left(\prescript{TR}{0}{\mathscr{D}}_{\Psi(t)}^{\alpha, \lambda}\right)^kf(t)\right) \right]_{t=0}.
\end{align*}
Summing this result over $k$ from $0$ to $n$, we obtain
\begin{multline} \label{844}
f(t)-\left(\prescript{T}{0}{\mathscr{I}}_{\Psi(t)}^{\alpha, \lambda}\right)^{n+1}\left(\prescript{TR}{0}{\mathscr{D}}_{\Psi(t)}^{\alpha, \lambda}\right)^{n+1}f(t) \\ =e^{-\lambda\Psi(t)}\,\sum_{k=0}^{n} \,\frac{(\Psi(t)-\Psi(0))^{(k+1)\alpha-1}}{\Gamma((k+1)\alpha)} \cdot\left[\prescript{}{0}{\mathscr{I}}_{\Psi(t)}^{1-\alpha}\left( e^{\lambda\Psi(t)}\left(\prescript{TR}{0}{\mathscr{D}}_{\Psi(t)}^{\alpha, \lambda}\right)^kf(t)\right) \right]_{t=0}.
\end{multline}
By Theorem \ref{t22}, there exists $c\in(0,t)$ such that
\begin{multline} \label{845}
\left(\prescript{T}{0}{\mathscr{I}}_{\Psi(t)}^{\alpha, \lambda}\right)^{n+1}\left(\prescript{TR}{0}{\mathscr{D}}_{\Psi(t)}^{\alpha, \lambda}\right)^{n+1}f(t) \\ =\left(\Psi(t)-\Psi(0)\right)^{(n+1)\alpha}e^{-\lambda(\Psi(t)-\Psi(0))} E_{1,(n+1)\alpha+1}\Big( \lambda(\Psi(t)-\Psi(0))\Big)\cdot \left(\prescript{TR}{0}{\mathscr{D}}_{\Psi(t)}^{\alpha, \lambda}\right)^{n+1}f(c).
\end{multline}
Using the equation \eqref{845} in  the equation \eqref{844}, we obtain the stated result.
\end{proof}

\section{Fractional differential equations} \label{Sec:fde}

As far as the matter of solving fractional differential equations is concerned, there are numerous possible methods and even different types of conclusions that can be drawn, from exact constructed solutions to approximate numerical solutions to qualitative properties of solutions to simplify the existence and uniqueness of solutions. In the current section, we will consider how such studies and results can be done in the setting of $\Psi$-tempered fractional differential equations.

\subsection{Existence and uniqueness of solutions}

Often, existence and uniqueness are proved by using fixed point theorems such as the Banach contraction principle, and a useful tool enabling us to apply such theorems is to find equivalences between fractional initial value problems and fractional integral equations. In the setting of $\Psi$-tempered fractional differential equations, we shall first prove two such equivalences, one for Riemann--Liouville type differential equations and one for Caputo type differential equations, and then proceed to prove existence-uniqueness theorems using the Banach contraction principle.


\begin{theorem}\label{c3}
	Let $\alpha\in(0,1)$, $\lambda\in\mathbb{R}$, and  $f:(0, b]\times \R\rightarrow\R$ be a function such that $f(\cdot\,,y(\cdot))\in C_{1-\alpha\,; \,\Psi}([0, b],\,\R)$, for every $y\in C_{1-\alpha\,; \,\Psi}([0, b],\,\R)$. Then, the $\Psi$-tempered fractional differential equation of Riemann--Liouville type
	\begin{align}
\prescript{TR}{0}{\mathscr{D}}_{\Psi(t)}^{\alpha, \lambda}\,y(t)&=f(t, y(t)),\quad t\in(0, b],\label{871}
	\end{align} 
	with initial condition
	\begin{align}
		\left[\prescript{T}{0}{\mathscr{I}}_{\Psi(t)}^{1-\alpha, \lambda}\left( e^{\lambda\Psi(t)}y(t)\right) \right] _{t=0}&=y_0\in\R, \label{872}
	\end{align} 
	is equivalent to the fractional integral equation
	\begin{equation}\label{873}
		y(t)= y_0\,e^{-\lambda\Psi(t)} \frac{(\Psi(t)-\Psi(0))^{\alpha-1}}{\Gamma(\alpha)} +  \prescript{T}{0}{\mathscr{I}}_{\Psi(t)}^{\alpha, \lambda} f(t, y(t)),\quad t\in(0, b],
	\end{equation}
	both to be solved for functions $y\in C_{1-\alpha\,; \,\Psi}([0, b],\,\R)$.
\end{theorem}

\begin{proof}
The corresponding result for $\Psi$-Riemann--Liouville fractional integro-differential equations is an immediate consequence of setting $ \beta=0 $ in the equivalence between \cite[Eq. (22)--(23)]{Sousa} and \cite[Eq. (26)]{Sousa}. Then, using this result together with the conjugation relations given by Proposition \ref{Conjugationrelations}, we obtain the required result.
\end{proof}

\begin{theorem}\label{c2}
Let $\alpha\in(0,1)$, $\lambda\in\mathbb{R}$, and $f\in C\left( [0, b]\times \R, \R\right) $. Then, the $\Psi$-tempered fractional differential equation of Caputo type
\begin{align}
\prescript{TC}{0}{\mathscr{D}}_{\Psi(t)}^{\alpha, \lambda}y(t)&=f(t, y(t)),\quad t\in(0, b],\label{865}
\end{align} 
with initial condition
\begin{align}
	y(0)&=y_0\in\R, \label{866}
\end{align} 
is equivalent to the fractional integral equation
\begin{equation}\label{867}
y(t)= e^{-\lambda(\Psi(t)-\Psi(0))} y_0 + \prescript{T}{0}{\mathscr{I}}_{\Psi(t)}^{\alpha, \lambda} f(t, y(t)),\quad t\in[0, b],
\end{equation}
both to be solved for functions $y\in AC_{\Psi}[0, b]$.
\end{theorem}

\begin{proof}
Using the result proved in \cite[Theorem~2]{Ricardo} together with the conjugation relations given by Proposition \ref{Conjugationrelations}, we obtain the required result.
\end{proof}


\begin{theorem}\label{53}
	Let $\alpha,\lambda,f$ be as in Theorem \ref{c3}, and further assume that $\alpha\in(\frac{1}{2},1)$, $\lambda\geq0$ and that $f$ satisfies the following Lipschitz condition with respect to the second variable, for some fixed $L>0$:
	\begin{equation}\label{901}
		\left| f(t, u)- f(t, v)\right| \leq L \left| u-v\right|,\quad u, v\in\R,\quad t\in[0, b].
	\end{equation}
	Then, the initial value problem \eqref{871}--\eqref{872} has a unique solution $y\in C_{1-\alpha;\,\Psi} ([0, b],\,\R)$, provided that the constant $L$ is less than
	\begin{equation}\label{904}
\frac{\Gamma(2\alpha)}{\Gamma(\alpha)}\left( \Psi(b)-\Psi(0)\right)^{-\alpha} e^{-\lambda\Psi(b)}.
	\end{equation}
\end{theorem}

	\begin{proof}
		Consider the operator  $\mathscr{B}:  C_{1-\alpha;\,\Psi} \left([0, b], \R \right) \to C_{1-\alpha;\,\Psi} \left([0, b], \R \right)$ defined by
		\begin{multline*}
			\mathscr{B}y(t)=(\Psi(t)-\Psi(0))^{\alpha-1}\,e^{-\lambda\Psi(t)}\frac{y_0}{\Gamma(\alpha)}\\
			+\frac{e^{-\lambda\Psi(t)}}{\Gamma(\alpha)}\int_{0}^{t} \Psi'(s) \left( \Psi(t)-\Psi(s)\right)^{\alpha-1}e^{\lambda\Psi(s)}f(s, y(s)) \,\mathrm{d}s,\quad t\in (0, b].
		\end{multline*}
		Then, the equation \eqref{873} can be written as $y=\mathscr{B}y$.
		
For any $u, v\in C_{1-\alpha;\,\Psi}[0, b]$ and $t\in[0, b]$, we have
		\begin{multline*}
\Big|(\Psi(t)-\Psi(0))^{1-\alpha}\left(  \mathscr{B}u(t)-\mathscr{B}v(t)\right) \Big|\\
\leq \frac{(\Psi(t)-\Psi(0))^{1-\alpha}e^{-\lambda\Psi(t)}}{\Gamma(\alpha)}\int_{0}^{t} \Psi'(s) \left( \Psi(t)-\Psi(s)\right)^{\alpha-1}e^{\lambda\Psi(s)}\left| f(s, u(s))-f(s, v(s)) \right| \,\mathrm{d}s.
		\end{multline*}
		Using the Lipschitz condition \eqref{901} along with the non-decreasing nature of $\Psi$ on $[0, b]$, we obtain
		\begin{align*}
			\Big|(\Psi(t)&-\Psi(0))^{1-\alpha}\left(  \mathscr{B}u(t)-\mathscr{B}v(t)\right) \Big|\\
			&\leq \frac{L(\Psi(b)-\Psi(0))^{1-\alpha}}{\Gamma(\alpha)}\int_{0}^{t} \Psi'(s) \left( \Psi(t)-\Psi(s)\right)^{\alpha-1}e^{\lambda\Psi(s)}(\Psi(s)-\Psi(0))^{\alpha-1}\\
			&\hspace{5cm}\times\Big| (\Psi(s)-\Psi(0))^{1-\alpha}\left(  u(s)-v(s)\right) \Big| \,\mathrm{d}s\\
			&\leq L(\Psi(b)-\Psi(0))^{1-\alpha}e^{\lambda\Psi(b)} \prescript{}{0}{\mathscr{I}}_{\Psi(t)}^{\alpha} \left( \Psi(t)-\Psi(0)\right)^{\alpha-1}\left\|u-v\right\|_{C_{1-\alpha;\,\Psi}[0, b]} \\
			&=L(\Psi(b)-\Psi(0))^{1-\alpha}e^{\lambda\Psi(b)} \cdot\frac{\Gamma(\alpha)}{\Gamma(2\alpha)}(\Psi(t)-\Psi(0))^{2\alpha-1}\left\|u-v\right\|_{C_{1-\alpha;\,\Psi}[0, b]}\\
			&\leq L\,\frac{\Gamma(\alpha)}{\Gamma(2\alpha)} e^{\lambda\Psi(b)}\left( \Psi(b)-\Psi(0)\right)^{\alpha}\left\|u-v\right\|_{C_{1-\alpha;\,\Psi}[0, b]}.
		\end{align*}
Note that the non-decreasing nature of $\Psi$ together with the condition $\alpha>\frac{1}{2}$ is used to conclude that $(\Psi(t)-\Psi(0))^{2\alpha-1}\leq(\Psi(b)-\Psi(0))^{2\alpha-1}$. Taking the supremum over $t\in(0,b]$, we obtain
		\[
		\left\| \mathscr{B}u- \mathscr{B}v\right\|_{C_{1-\alpha;\,\Psi}[0, b]} \leq L\,\frac{\Gamma(\alpha)}{\Gamma(2\alpha)} e^{\lambda\Psi(b)}\left( \Psi(b)-\Psi(0)\right)^{\alpha}\left\|u-v\right\|_{C_{1-\alpha;\,\Psi}[0, b]}.
		\]
If $L$ is less than the expression \eqref{904}, this means $\mathscr{B}$ is a contraction. Then, by the Banach contraction principle, the initial value problem \eqref{871}--\eqref{872} has a unique solution $y\in C_{1-\alpha;\,\Psi}[0, b]$.
	\end{proof}

\begin{theorem}\label{thm 54}
If $\alpha, f$ are as in Theorem \ref{c2}, $\lambda\geq0$ and $f$ satisfies the Lipschitz condition \eqref{901}, then the initial value problem \eqref{865}--\eqref{866} has a unique solution $y\in AC_{\Psi}[0, b]$, provided that $L$ is less than
\begin{equation}\label{902}
\Gamma(\alpha+1)\left( \Psi(b)-\Psi(0)\right)^{-\alpha} e^{-\lambda\left( \Psi(b)-\Psi(0)\right)}.
\end{equation}
\end{theorem}

\begin{proof}
Consider the Banach space of continuous functions $C = AC_\Psi[0, b]$, endowed with the supremum norm $\left\|\cdot\right\|_\infty$, and define the operator  $\mathscr{B}:  C \to C$ by 
\begin{multline*}
\mathscr{B}y(t)=e^{-\lambda(\Psi(t)-\Psi(0))} y_0\\
+\frac{e^{-\lambda\Psi(t)}}{\Gamma(\alpha)}\int_{0}^{t} \Psi'(s) \left( \Psi(t)-\Psi(s)\right)^{\alpha-1}e^{\lambda\Psi(s)}f(s, y(s)) \,\mathrm{d}s,\quad t\in [0, b].
\end{multline*}
Then, the equation \eqref{867} can be written as $y=\mathscr{B}y$.

For any $u, v\in C$ and $t\in[0, b]$, we have
\[
\big| \mathscr{B}u(t)-\mathscr{B}v(t)\big| \leq \frac{e^{-\lambda\Psi(t)}}{\Gamma(\alpha)}\int_{0}^{t} \Psi'(s) \left( \Psi(t)-\Psi(s)\right)^{\alpha-1}e^{\lambda\Psi(s)}\big| f(s, u(s))-f(s, v(s)) \big|\,\mathrm{d}s.
\]
Using the Lipschitz condition \eqref{901} along with the non-decreasing nature of $\Psi$ on $[0, b]$, we obtain
\begin{align*}
\big| \mathscr{B}u(t)-\mathscr{B}v(t)\big|
&\leq \frac{L}{\Gamma(\alpha)}\int_{0}^{t} \Psi'(s) \left( \Psi(t)-\Psi(s)\right)^{\alpha-1}e^{\lambda\left( \Psi(s)-\Psi(0)\right)}\big|u(s)-v(s)\big|\,\mathrm{d}s \\
&\leq\left\|u-v\right\|_{\infty}\frac{L\, e^{\lambda\left( \Psi(b)-\Psi(0)\right)}}{\Gamma(\alpha)}\int_{0}^{t} \Psi'(s) \left( \Psi(t)-\Psi(s)\right)^{\alpha-1} \,\mathrm{d}s\\
&=\left\|u-v\right\|_{\infty}\frac{L\, e^{\lambda\left( \Psi(b)-\Psi(0)\right)}}{\Gamma(\alpha+1)}\left( \Psi(t)-\Psi(0)\right)^\alpha\\
&\leq\frac{L\, e^{\lambda\left( \Psi(b)-\Psi(0)\right)}}{\Gamma(\alpha+1)}\left( \Psi(b)-\Psi(0)\right)^\alpha\left\|u-v\right\|_{\infty}.
\end{align*}
Taking the supremum over $t\in(0,b]$, this gives
\[
\left\| \mathscr{B}u- \mathscr{B}v\right\|_\infty \leq L\frac{\, e^{\lambda\left( \Psi(b)-\Psi(0)\right)}\left( \Psi(b)-\Psi(0)\right)^\alpha}{\Gamma(\alpha+1)}\left\|u-v\right\|_{\infty}.
\]
If $L$ is less than the expression \eqref{902}, this means $\mathscr{B}$ is a contraction. Then, by the Banach contraction principle, the initial value problem \eqref{865}--\eqref{866} has a unique solution $y\in AC_\Psi[0, b]$.
\end{proof}

\subsection{Gr\"onwall's inequality and stability results}

\begin{theorem}\label{uhr}
	Let $\alpha\in(0,\infty)$ and $ \lambda \in \R $. Assume that $u,v\in L^1([0,b],\mathrm{d}\Psi)$ are non-negative and $ w:[0,b]\to\mathbb{R}^+_0=[0, \infty)$ is non-negative, non-decreasing, and continuous. If
	\begin{equation*}
		u(t) \leq v(t) + w(t)\cdot \int_{0}^{t} \Psi'(s)\left(\Psi(t)-\Psi(s)\right)^{\alpha -1}e^{-\lambda\left(\Psi(t)-\Psi(s)\right)}\,u\left( s\right) \,\mathrm{d}s,\quad t\in[0,b],
	\end{equation*}
	then
	\begin{equation*}
		u(t) \leq v(t) +  \int_{0}^{t} \sum_{k=1}^{\infty}\frac{[{w(t) \Gamma(\alpha)}]^{k}}{\Gamma(\alpha k)}\Psi'(s)\left(\Psi(t)-\Psi(s)\right)^{\alpha k -1}e^{-\lambda\left(\Psi(t)-\Psi(s)\right)}\,v\left( s\right) \,\mathrm{d}s, \quad t \in [0,b].
	\end{equation*}
\end{theorem}

\begin{proof}
If we consider $ \rho=1 $ in \cite[Theorem 11]{FahadMujeeb}, we recover the following Gr\"onwall's inequality for non-negative $p,q\in L^1[a,B]$ and non-negative non-decreasing continuous $g:[a,B]\to\mathbb{R}^+_0$: if
\begin{equation}
p(t) \leq q(t) + \Gamma(\alpha)\cdot g(t)\cdot\prescript{T}{a}{\mathscr{I}}_{t}^{\alpha,\lambda}p(t),\quad t\in[a,B],\label{gronfortemperedif}
\end{equation}
then
\begin{equation}
p(t) \leq q(t) + \sum_{k=1}^{\infty}\big[\Gamma(\alpha)\cdot g(t)\big]^{k}\cdot\prescript{T}{a}{\mathscr{I}}_{t}^{\alpha k,\lambda}q(t), \quad t \in [a,B].\label{gronfortemperedthen}
\end{equation}
Defining $ p= Q_{\Psi}^{-1} u$, $ q=Q_{\Psi}^{-1} v $, $ g = Q_{\Psi}^{-1} w $, and also $a=\Psi(0)$ and $B=\Psi(b)$, in equations \eqref{gronfortemperedif}--\eqref{gronfortemperedthen}, and using the fact that
\[
\prescript{T}{a}{\mathscr{I}}_{t}^{\alpha,\lambda}\,p = \left( \prescript{T}{a}{\mathscr{I}}_{t}^{\alpha,\lambda}\, \circ Q_{\Psi}^{-1}\right)u = \left( Q_{\Psi}^{-1} \circ \prescript{T}{0}{\mathscr{I}}_{\Psi(t)}^{\alpha,\lambda} \right) u = \left(\prescript{T}{0}{\mathscr{I}}_{\Psi(t)}^{\alpha,\lambda}  u \right) \circ \Psi ^{-1},
\]
we get the required result by using the conjugation relations given in Proposition \ref{Conjugationrelations}.
\end{proof}

\begin{rem}
In the special case $\lambda=0$, we obtain the result of \cite[Theorem 3]{Sousa}. Further, if we assume $v$ is a non-decreasing function, we obtain the following result (which is \cite[Corollary 2]{Sousa} in the case $\lambda=0$):
\begin{equation} \label{Gronwall:coroll}
	u(t) \leq v(t) \,E_\alpha\left(  w(t)\Gamma(\alpha)\left(\Psi(t)-\Psi(0)\right)^{\alpha}\right), \quad  t \in [0,b].
\end{equation}
\end{rem}

We can use the above result to establish some Ulam type stabilities for the equation \eqref{865}. Firstly, we write down some required definitions, following Wang et al. \cite{Wang}, for the various stabilities to be investigated.

\begin{definition}\label{d1}
Let $\alpha\in(0,1)$, $\lambda\in\mathbb{R}$, and $f\in C\left( [0, b]\times \R, \R\right) $.

The equation \eqref{865} is said to be Ulam--Hyers stable if there exists a constant $\phi_f > 0 $, such that for each $\varepsilon > 0 $ and for each $z \in AC_\Psi[0, b]$ satisfying the inequality
\begin{equation}\label{ed1}
\left| \prescript{TC}{0}{\mathscr{D}}_{\Psi(t)}^{\alpha, \lambda}z(t)-f\left(t, z(t)\right)\right| \leq \varepsilon, \qquad t\in (0, b],
\end{equation} 
there exists a solution $y \in AC_\Psi[0, b]$ of the equation \eqref{865} with
\[
\left\| z- y\right\|_\infty\leq \phi_f \,\varepsilon.
\]

The equation \eqref{865} is said to be generalised Ulam--Hyers stable if there exists a function $\xi_f \in C(\R^+_0,\R^+_0)$ with $\xi_f (0) = 0$, such that for each $\varepsilon > 0 $ and for each $z \in AC_\Psi[0, b]$ satisfying the inequality \eqref{ed1}, there exists a solution $y \in AC_\Psi[0, b]$ of the equation \eqref{865} with
\[
\left\| z- y\right\|_\infty \leq \xi_f (\varepsilon).
\]
\end{definition}

\begin{definition}\label{d3}
Let $\alpha\in(0,1)$, $\lambda\in\mathbb{R}$, and $f\in C\left( [0, b]\times \R, \R\right) $.

The equation \eqref{865} is said to be Ulam--Hyers--Rassias stable, with respect to a given function $\Omega\in C([0, b],\,\R_+)$,  if there exists a constant $\phi_{f,\Omega} > 0 $ such that for each $\varepsilon > 0 $ and for each $z \in AC_\Psi[0, b]$ satisfying the inequality
\begin{equation}\label{ed3}
\left| \prescript{TC}{0}{\mathscr{D}}_{\Psi(t)}^{\alpha, \lambda}z(t)-f(t, z(t))\right| \leq \varepsilon\,\Omega(t), \qquad t\in (0, b],
\end{equation}
there exists a solution $y \in AC_\Psi[0, b]$ of the equation \eqref{865} with
\[
\left| z(t)-y(t)\right|  \leq  \,\varepsilon\,\phi_{f,\Omega} \,\Omega(t),\qquad t\in [0, b].
\]

The equation \eqref{865} is said to be generalised Ulam--Hyers--Rassias stable, with respect to a given function $\Omega\in C([0, b],\,\R_+)$, if there exists a function $C_{f, \Omega} > 0 $ such that for each $\varepsilon>0$ and for each $z \in AC_\Psi[0, b]$ satisfying the inequality
\begin{equation}\label{ed4}
\left| \prescript{TC}{0}{\mathscr{D}}_{\Psi(t)}^{\alpha, \lambda}z(t)-f(t, z(t))\right| \leq \,\Omega(t), \qquad t\in (0, b],
\end{equation} 
there exists a solution $y \in AC_\Psi[0, b]$ of the equation \eqref{865} with
\[
\left| z(t)-y(t)\right| \leq  \,\phi_{f,\Omega} \,\Omega(t),\qquad t\in [0, b].
\]
\end{definition}


\begin{theorem} \label{th6.1}
If $\alpha,\lambda,f$ are as in Theorem \ref{thm 54}, and if $\Omega\in C([0, b],\,\R_+)$ is a non-decreasing function and $K > 0$ is a constant such that
\begin{equation}\label{am}
\prescript{T}{0}{\mathscr{I}}_{\Psi(t)}^{\alpha, \lambda}\Omega(t) \leq K\,\Omega(t),\qquad t \in (0, b],
\end{equation}
then the fractional differential equation \eqref{865} is Ulam--Hyers--Rassias stable with respect to $\Omega$.
\end{theorem}

\begin{proof}
Fix $\varepsilon>0$, and let $z \in AC_\Psi[0, b]$ be any solution of  the inequality 
\begin{equation}\label{5.1a}
\left| \prescript{TC}{0}{\mathscr{D}}_{\Psi(t)}^{\alpha, \lambda}z(t)-f(t, z(t))\right| \leq \varepsilon\,\Omega(t),\qquad t \in (0, b].
\end{equation}
Then, we can define $ w\in AC_\Psi[0, b]$ such that
\begin{equation}\label{5.7}
\prescript{TC}{0}{\mathscr{D}}_{\Psi(t)}^{\alpha, \lambda}z(t)=f\left(t, z(t)\right)+w(t),\qquad t\in[0,b],
\end{equation}
and $\left| w(t)\right|\leq \varepsilon\,\Omega(t)$, for all $t\in (0, b]$. Now, let $y \in AC_\Psi[0, b]$ be a solution of the following initial value problem:
\begin{equation}\label{5.2a1}
\begin{cases}
 \prescript{TC}{0}{\mathscr{D}}_{\Psi(t)}^{\alpha, \lambda}y(t)=f(t, y(t)), \qquad t \in (0, b],\\
y(0)=z(0),
\end{cases}
\end{equation}
which we know exists (uniquely) by Theorem \ref{thm 54}. Also, by Theorem \ref{c2}, the equivalent fractional integral equation to \eqref{5.2a1} is
\begin{equation}\label{5.3}
y(t)= e^{-\lambda(\Psi(t)-\Psi(0))} z(0) + \prescript{T}{0}{\mathscr{I}}_{\Psi(t)}^{\alpha, \lambda} f(t, y(t)),\qquad t\in [0, b],
\end{equation}
while the solution of the equation \eqref{5.7} is given by
\begin{equation}\label{59}
z(t)= e^{-\lambda(\Psi(t)-\Psi(0))} z(0) + \prescript{T}{0}{\mathscr{I}}_{\Psi(t)}^{\alpha, \lambda} f(t, z(t))+\prescript{T}{0}{\mathscr{I}}_{\Psi(t)}^{\alpha, \lambda}w(t),\qquad t\in [0, b].
\end{equation}
From equation \eqref{59} and inequality \eqref{am}, we have
\begin{align*}
\left|z(t)- e^{-\lambda(\Psi(t)-\Psi(0))} z(0) - \prescript{T}{0}{\mathscr{I}}_{\Psi(t)}^{\alpha, \lambda} f(t, z(t))\right| 
&\leq \prescript{T}{0}{\mathscr{I}}_{\Psi(t)}^{\alpha, \lambda}\left| w(t)\right| \\
&\leq \varepsilon\,\prescript{T}{0}{\mathscr{I}}_{\Psi(t)}^{\alpha, \lambda}\Omega(t) \\
&\leq \varepsilon \,K\,\Omega(t).
\end{align*}
Using this together \eqref{5.3} and the Lipschitz condition \eqref{901}, for each $t\in [0, b]$ we have
\begin{align*}
\big|z(t)&-y(t)\big|\\
&=\left|z(t)-\left[ e^{-\lambda(\Psi(t)-\Psi(0))} z(0) + \prescript{T}{0}{\mathscr{I}}_{\Psi(t)}^{\alpha, \lambda} f(t, y(t))\right] \right|\\
&\leq\left|z(t)- e^{-\lambda(\Psi(t)-\Psi(0))} z(0) -\prescript{T}{0}{\mathscr{I}}_{\Psi(t)}^{\alpha, \lambda} f(t, z(t))\right| +\left| \prescript{T}{0}{\mathscr{I}}_{\Psi(t)}^{\alpha, \lambda} f(t, z(t))-\prescript{T}{0}{\mathscr{I}}_{\Psi(t)}^{\alpha, \lambda} f(t, y(t)) \right|\\
&\leq \varepsilon \,K\,\Omega(t)+\frac{1}{\Gamma(\alpha)}\int_{a}^{t}\Psi'(s)(\Psi(t)-\Psi(s))^{\alpha-1}e^{-\lambda(\Psi(t)-\Psi(s))}\big|f\left( s, z(s)\right)-f\left( s, y(s)\right) \big|\,\mathrm{d}s\\
&\leq \varepsilon \,K\,\Omega(t) + L\,\frac{1}{\Gamma(\alpha)}\int_{a}^{t}\Psi'(s)(\Psi(t)-\Psi(s))^{\alpha-1}\left| z(s)- y(s) \right| \,\mathrm{d}s.
\end{align*}
Applying the corollary \eqref{Gronwall:coroll} of Gr\"onwall's inequality, with $u(t)=\left|z(t)-y(t) \right|$ and $v(t)=\varepsilon \,K\,\Omega(t)$ and $w(t)=\frac{L}{\Gamma(\alpha)}$, we obtain
\begin{align}\label{62}
\left| z(t)-y(t)\right| 
&\leq\varepsilon\,K\,\Omega(t)\,E_{\alpha}\Big(L\left(\Psi(t)-\Psi(0)\right)^{\alpha}\Big)\nonumber\\
&\leq \varepsilon\,\phi_{f,\Omega} \,\Omega(t),\quad t\in [0, b],
\end{align}
where $\phi_{f,\Omega}:=K\,E_{\alpha}\left(L\left(\Psi(b)-\Psi(0)\right)^{\alpha}\right)$. Thus, the equation \eqref{865} is Ulam--Hyers--Rassias stable with respect to $\Omega$.
\end{proof}

\begin{cor}
Assuming the same conditions as in Theorem \ref{th6.1}, we have the following immediate consequences of the above result on Ulam--Hyers--Rassias stability.
\begin{enumerate}
\item Putting $\varepsilon=1$, we get that the equation \eqref{865} is generalised Ulam--Hyers--Rassias stable.
\item Putting $\Omega(t)=1$, we get that the equation \eqref{865} is  Ulam--Hyers stable.
\item Putting $\Omega(t)=1$ and $\xi_f (\varepsilon)=\varepsilon\, \phi_{f,1}$, we get that the equation \eqref{865} is generalised Ulam--Hyers stable.
\end{enumerate}
\end{cor}

\section{Conclusions and future directions}

In this paper, we have continued the development of tempered fractional calculus with respect to functions, which was first introduced in a previous paper of the last two authors with their collaborators \cite{Fahad}  and which forms the unique intersection of two general classes of operators (fractional calculus with general analytic kernels with respect to functions, and weighted fractional calculus with respect to functions) as well as a unique extension of two useful types of operators (those of tempered fractional calculus and those of fractional calculus with respect to functions) both of which have found applications in the study of continuous time random walks.

The results of the current paper might form the theoretical foundations for further developing and studying the fractional calculus of tempered operators of a function with respect to another function, as well as for solving and analysing fractional differential equations involving these operators. Some nonlinear ordinary differential equations have already been studied here, with existence-uniqueness theorems and stability theorems being established, but further extensions of these studies are possible, such as proving qualitative properties and bounds, as well as constructing explicit solutions in some special cases by other methods. Partial differential equations involving $\Psi$-tempered operators could also be studied.

Another natural extension would be to consider the calculus of tempered Hilfer fractional derivatives with respect to functions. These operators would be related to the classical Hilfer derivatives \cite{Hilfer} by conjugation relations exactly analogous to those in Proposition \ref{Conjugationrelations}, so many results concerning them could be proved very easily. Hilfer-type $\Psi$-tempered fractional differential equations could also be studied, and some results pertaining to existence, uniqueness, and stability for such equations will be published in a forthcoming paper along with the basic theory of the aforesaid operators.

\section*{Acknowledgment}
The second author acknowledges the Science and Engineering Research Board (SERB), New Delhi, India for the Research Grant (Ref: File no. EEQ/2018/000407). The fourth author would like to thank the Isaac Newton Institute for Mathematical Sciences, Cambridge, for support and hospitality during the programme Fractional Differential Equations where work on this paper was undertaken. This work was supported by EPSRC grant no. EP/R014604/1.

\section*{Declaration of interests}
The authors declare that they have no known competing financial interests or personal relationships that could have appeared to influence the work reported in this paper.


\end{document}